\documentclass[11pt,reqno]{amsart}
\usepackage{amsmath}
\usepackage{amstext}
\usepackage{amsbsy}
\usepackage{amsopn}
\usepackage{upref}
\usepackage{amsthm}
\usepackage{amsfonts}
\usepackage{amssymb}
\usepackage{mathrsfs}
\usepackage{times}
\allowdisplaybreaks
 \newtheorem{theorem}{Theorem}
 
 \newtheorem{proposition}[theorem]{Proposition}
 \newtheorem{lemma}[theorem]{Lemma}
\theoremstyle{definition}
 
\theoremstyle{remark}
 
\begin{document}
\title[Dispersive Flow]{A fourth-order dispersive flow into K\"ahler manifolds}
\author[H.~Chihara and E.~Onodera]{Hiroyuki Chihara and Eiji Onodera}
\address[HC]{Department of Mathematics and Computer Science, Kagoshima University, Kagoshima 890-0065, Japan}
\email{chihara@sda.att.ne.jp}
\address[EO]{Department of Mathematics, Kochi University, Kochi 780-8520, Japan}
\email{onodera@kochi-u.ac.jp}
\thanks{HC is supported by JSPS Grant-in-Aid for Scientific Research \#23340033.}
\thanks{EO is supported by JSPS Grant-in-Aid for Scientific Research \#24740090.}
\begin{abstract}
We discuss a short-time existence theorem of solutions to 
the initial value problem for a fourth-order dispersive flow 
for curves parametrized by the real line into a compact K\"ahler manifold. 
Our equations geometrically generalize a physical model 
describing the motion of a vortex filament or  
the continuum limit of the Heisenberg spin chain system. 
Our results are proved by using so-called the energy method.  
We introduce a bounded gauge transform on the pullback bundle, 
and make use of local smoothing effect of the dispersive flow a little. 
\end{abstract}
\maketitle
\section{Introduction}
\label{section:introduction} 
Let $(N,J,g)$ be a compact K\"ahler manifold of real dimension $2n$ 
with a complex structure $J$ and a K\"ahler metric $g$. 
In the present paper we study the initial value problem for a mapping 
$\mathbb{R}\times\mathbb{R}\ni(t,x) \mapsto u(t,x){\in}N$ of the form 
\begin{alignat}{2}
& u_t
  =
  aJ(u)\nabla_x^3u_x
  +
  \{1+bg_u(u_x,u_x)\}J(u)\nabla_xu_x
  +
  cg_u(\nabla_xu_x,u_x)J(u)u_x
& \quad
& \text{in}
  \quad
  \mathbb{R}\times\mathbb{R},
\label{equation:pde}
\\
& u(0,x)
  =
  u_0(x)
& \quad
& \text{in}
  \quad
  \mathbb{R},
\label{equation:data}
\end{alignat}
where 
$u_t=du(\partial/\partial t)$, 
$u_x=du(\partial/\partial x)$, 
$du$ is the differential of the mapping $u$, 
$\nabla$ is the induced connection for the Levi-Civita connection $\nabla^N$ of $(N,J,g)$, 
$a\in\mathbb{R}\setminus\{0\}$ and $b,c\in\mathbb{R}$ are constants, 
and 
$u_0:\mathbb{R} \rightarrow N$ is a given initial curve on $N$. 
$u(t,\cdot)$ is a curve on $N$ for any fixed $t\in\mathbb{R}$, 
and 
$u$ describes the motion of a curve subject to the equation \eqref{equation:pde}. 
If $a,b,c=0$, then \eqref{equation:pde} is reduced to 
the one-dimensional Schr\"odinger map equation of the form 
$$
u_t=J(u)\nabla_xu_x,
$$
and its stationary solutions are geodesics on $N$. 
See, e.g., 
\cite{CSU}, 
\cite{chihara}, 
\cite{darios}, 
\cite{hasimoto}, 
\cite{koiso} 
and references therein 
for the physical background and the mathematical study of the Schr\"odinger map equation. 
\par
Here we present local expression of the covariant derivative $\nabla_x$. 
Let $y^1,\dotsc,y^{2n}$ be local coordinates of $N$. 
We denote by $\Gamma^\alpha_{\beta\gamma}$, $\alpha,\beta,\gamma=1,\dotsc,2n$,  
the Christoffel symbol of $(N,J,g)$. 
For a smooth curve $u:\mathbb{R}\rightarrow{N}$, 
$\Gamma(u^{-1}TN)$ is the set of all smooth sections 
of the pullback bundle $u^{-1}TN$. 
If we express $V{\in}\Gamma(u^{-1}TN)$ as
$$
V(x)=\sum_{\alpha=1}^{2n}V^\alpha(x)\left(\frac{\partial}{\partial{y^\alpha}}\right)_u, 
$$
then, $\nabla_xV$ is given by  
$$
\nabla_xV(x)
=
\sum_{\alpha=1}^{2n}
\left\{
\frac{{\partial}V^\alpha}{\partial{x}}(x)
+
\sum_{\beta,\gamma=1}^{2n}
\Gamma^\alpha_{\beta\gamma}\bigl(u(x)\bigr)V^\beta(x)\frac{\partial{u^\gamma}}{\partial{x}}(x)
\right\}
\left(\frac{\partial}{\partial{y^\alpha}}\right)_u.
$$
\par
Here we recall the physical background of \eqref{equation:pde}. 
The equation generalizes a model system arising in classical mechanics of the form 
\begin{equation}
\vec{u}_t
=
\vec{u}
\times
\bigl\{
a\vec{u}_{xxxx}
+
\vec{u}_{xx}
+
b\langle\vec{u}_x,\vec{u}_x\rangle\vec{u}_{xx}
+
c\langle\vec{u}_{xx},\vec{u}_x\rangle\vec{u}_x
\bigr\},
\label{equation:model}
\end{equation}
where $\mathbb{R}\times\mathbb{R}\ni(t,x) \mapsto \vec{u}(t,x)\in\mathbb{S}^2$, 
$\mathbb{S}^2$ is the two-dimensional unit sphere in $\mathbb{R}^3$, 
and 
$\vec{\xi}\times\vec{\eta}$ 
and 
$\langle\vec{\xi},\vec{\eta}\rangle$ 
are 
the vector and the inner products of 
$\vec{\xi}\in\mathbb{R}^3$ 
and 
$\vec{\eta}\in\mathbb{R}^3$ 
respectively. 
The equation \eqref{equation:model} describes the motion of a  vortex filament  
or the continuum limit of the Heisenberg spin chain system. 
In \cite{PDL} Porsezian, Daniel and Lakshmanan formulated 
the continuum limit of the Heisenberg spin chain system 
as \eqref{equation:model} with $3a+c=2b$ and $a\ne0$, 
and showed that if $c=0$ which is equivalent to $b=3a/2$, 
then \eqref{equation:model} is completely integrable. 
In \cite{fukumoto} Fukumoto studied the motion of vortex filament 
and obtained \eqref{equation:model} with with $3a+c=2b$ and $a\ne0$. 
More precisely, 
in \cite{fukumoto} $x$ describes the length of a curve $\vec{\gamma}(t,x)$, 
and $\vec{u}(t,x)=\vec{\gamma}_x(t,x)$. 
\par
It seems to be natural and fundamental to study the initial value problem 
\eqref{equation:pde}-\eqref{equation:data}. 
In particular, it is important to know the relationship between 
the geometric settings and the existence theorems of time-local and time-global solutions. 
The main difficulty of solving \eqref{equation:pde}-\eqref{equation:data} 
is the loss of derivative of order one occurring in \eqref{equation:pde} 
even in the case $N=\mathbb{S}^2$. 
If $(N,J,g)$ is a compact almost Hermitian manifold, 
then the situation becomes more difficult. 
Indeed, the loss of derivative of order three may happen 
since $\nabla^NJ$ does not necessarily vanish. 
\par
The known results on \eqref{equation:pde}-\eqref{equation:data} are limited. 
In particular, all the preceding results are concerned only with the case of $N=\mathbb{S}^2$. 
By the stereographic projection of $\mathbb{S}^2\setminus\{(0,0,1)\}$ onto $\mathbb{C}$, 
the equation \eqref{equation:model} becomes 
a complex-valued fourth-order semilinear dispersive partial differential equation. 
The time-local existence for this equation was established. 
See Huo and Jia (\cite{HJ1}, \cite{HJ2}, \cite{HJ3}), and Segata (\cite{segata1}, \cite{segata2}). 
The idea of proof of these results are based on 
making use of so-called the local smoothing effect of 
$\exp(\sqrt{-1}t\partial^4/\partial x^4)$ 
via Bourgain's Fourier restriction norm method. 
See \cite{bourgain} for this. 
The stereographic projection requires somewhat restriction on the range of the maps. 
For this reason, these results are not necessarily concerned 
with the study of maps to $\mathbb{S}^2$. 
On the other hand, 
Guo, Zeng and Su (\cite{GZS}) studied time-local existence of weak solutions 
for the model equation \eqref{equation:model} with
$a\ne0$, 
$c=0$ 
and 
$b=3a/2$ 
for  
$x \in \mathbb{R} / \mathbb{Z}$. 
In this case no smoothing effect can be expected 
since the source of the maps $\mathbb{R} / \mathbb{Z}$ is compact and 
singularities of solutions come back periodically. 
Fortunately, however, 
the integrability condition works well to overcome the loss of derivative of order one. 
\par
The purpose of the present paper is to establish the short-time existence theorem for 
the initial value problem \eqref{equation:pde}-\eqref{equation:data} 
from the point of view of geometric analysis 
(\cite{GR}, \cite{gunther}, \cite{hebey}, \cite{nishikawa}) 
and higher order linear dispersive partial differential equations 
(\cite{mizuhara}, \cite{tarama}). 
The present paper is a continuation of our geometric analysis of dispersive flows 
(\cite{chihara}, \cite{CO}, \cite{onodera1}, \cite{onodera2}, \cite{onodera3}). 
The point of view of geometric analysis sometimes offers deep insights 
into the structure of dispersive systems, and leads one to discoveries. 
For example, 
Koiso (\cite{koiso}) proved that some curvature condition on the target manifold guarantees 
the time-global existence for the Schr\"odinger flow for closed curves. 
Onodera (\cite{onodera1}) also established time-global existence theorem 
for a third-order dispersive flow for closed curves 
under some curvature condition on the target manifold. 
Moreover, Chihara (\cite{chihara}) came to understand the relationship between 
the K\"ahler condition of the target manifold and the structure of the Schr\"odinger map equation 
of a closed Riemannian manifold to a compact almost Hermitian manifold. 
\par
Here we introduce some function spaces of mappings. 
For $k=0,1,2,3,\dotsc$, 
we denote by $H^{k+1}(\mathbb{R};TN)$ the set of all continuous mappings of $\mathbb{R}$ to $N$ 
whose derivatives up to $k+1$ are all square integrable: 
$$
\lVert{u_x}\rVert_{H^{k}}^2
=
\sum_{l=0}^k
\int_{\mathbb{R}}
g_{u(x)}\bigl(\nabla_x^lu_x,\nabla_x^lu_x\bigr)
dx
<
\infty.
$$
The standard $k$-th order Sobolev space of $\mathbb{R}^d$-valued functions $\vec{z}$ 
on $\mathbb{R}$ is denoted by $H^k(\mathbb{R};\mathbb{R}^d)$, and its norm is defined by  
$$
\lVert{U}\rVert_{H^k(\mathbb{R};\mathbb{R}^d)}^2
=
\sum_{l=0}^k
\int_{\mathbb{R}}
\left\lvert
\frac{\partial^lU}{\partial x^l}
\right\rvert^2
dx,
$$ 
where $\lvert{U}\rvert=\sqrt{\langle{U,U}\rangle}$, 
and 
$\langle{U,V}\rangle$ is the standard inner product for 
$U, V \in \mathbb{R}^d$. 
Set 
$H^k(\mathbb{R})=H^k(\mathbb{R};\mathbb{C})$ 
and 
$L^2(\mathbb{R})=H^0(\mathbb{R};\mathbb{C})$ 
for short. 
Let $I$ be an interval in $\mathbb{R}$, and let $X$ be an appropriate function space.  
We denote by $C(I;X)$ the set of all $X$-valued continuous functions on $I$, 
and 
by $L^\infty(I;X)$ the set of all $X$-valued essentially bounded functions on $I$. 
\par
The main results of the present paper are the following. 
\begin{theorem}
\label{theorem:main} 
Let $k$ be an integer not smaller than six. 
For any initial mapping $u_0 \in H^{k+1}(\mathbb{R};TN)$, 
there exists a positive number $T$ depending only on 
$\lVert{u_{0x}}\rVert_{H^4(\mathbb{R};TN)}$ 
such that the initial value problem \eqref{equation:pde}-\eqref{equation:data} 
has a unique solution $u \in C\bigl([-T,T];H^{k+1}(\mathbb{R};TN)\bigr)$.  
\end{theorem}
Here we remark about the order of the smoothness of solutions 
required in Theorem~\ref{theorem:main}. 
The condition $k\geqslant6$ is determined by the Sobolev embedding 
$H^1(\mathbb{R}) \subset C(\mathbb{R}) \cap L^\infty(\mathbb{R})$. 
In our construction of solutions, 
we need $k\geqslant4$ 
for the boundedness of the both hand sides of the equation \eqref{equation:pde}. 
Our proof of uniqueness of solutions requires $k\geqslant6$ 
for the boundedness of the second derivative of \eqref{equation:pde}. 
\par
We shall prove Theorem~\ref{theorem:main} 
by using so-called the parabolic regularization 
and the uniform energy estimates of solutions to the regularized problems. 
In the latter part we introduce a gauge transform on the pullback bundle to overcome 
loss of derivative of order one. 
In other words, 
we slightly make use of local smoothing effect of dispersive equations 
via the gauge transform. 
In terms of pseudodifferential calculus, 
this consists of the identity and a pseudodifferential operator of order $-1$, 
The commutator between the lower order term of the gauge transform 
and the principal part of the equation becomes a second-order elliptic operator 
absorbing the loss of one derivative. 
This idea was actually applied to solving a third-order dispersive flow 
in \cite{CO} and \cite{onodera3}. 
\par
The plan of the present paper is as follows. 
In Section~\ref{section:mizuhara} 
we pick up a one-dimensional fourth-order 
linear dispersive partial differential equation, 
and illustrate the local smoothing effect and the gauge transform.  
Section~\ref{section:parabolic} is concerned with solving regularized problem. 
Section~\ref{section:energy} is devoted to the construction of a time-local solution. 
Finally, we shall prove the uniqueness of solution 
and recover the continuity in time of solution 
in Section~\ref{section:recovery}. 
\section{An Auxiliary Linear Problem}
\label{section:mizuhara}
The purpose of this section is to illustrate our idea of proof of Theorem~\ref{theorem:main}. 
Consider the initial value problem of the form 
\begin{alignat}{2}
& u_t
  =
  \sqrt{-1}au_{xxxx}
  +
  \sqrt{-1}\{\beta(t,x)u_x\}_x
  +
  \gamma(t,x)u_x
& \quad
& \text{in}
  \quad
  \mathbb{R}\times\mathbb{R},
\label{equation:pde1}
\\
& u(0,x)
  =
  u_0(x)
& \quad
& \text{in}
  \quad
  \mathbb{R},
\label{equation:data1}
\end{alignat}
where 
$u(t,x)$ is a complex-valued unknown function of $(t,x) \in \mathbb{R}\times\mathbb{R}$, 
$a \in \mathbb{R}\setminus\{0\}$ is a constant, 
$\beta(t,x) \in C\bigl(\mathbb{R};\mathscr{B}^\infty(\mathbb{R})\bigr)$ is real-valued, 
$\gamma(t,x) \in C\bigl(\mathbb{R};\mathscr{B}^\infty(\mathbb{R})\bigr)$ is complex-valued, 
$\mathscr{B}^\infty(\mathbb{R})$ 
is the set of all bounded smooth functions on $\mathbb{R}$ 
whose derivatives of any order are all bounded, 
and 
$u_0(x)$ is an initial data. 
We shall prove the following. 
\begin{proposition}
\label{theorem:aux}
Suppose that there exists a function $\phi(x) \in \mathscr{B}^\infty(\mathbb{R})$ such that 
$$
\lvert
\operatorname{Im} \gamma(t,x)
\rvert
\leqslant
\phi(x),
\quad
(t,x)\in\mathbb{R}^2,
\qquad
\int_{\mathbb{R}}
\phi(x)
dx
<\infty. 
$$
Then, the initial value problem 
{\rm \eqref{equation:pde1}-\eqref{equation:data1}} is $L^2$-well-posed, that is, 
for any $u_0 \in L^2(\mathbb{R})$, 
{\rm \eqref{equation:pde1}-\eqref{equation:data1}} has a unique solution 
$u \in C\bigl(\mathbb{R};L^2(\mathbb{R})\bigr)$.  
\end{proposition}
Roughly speaking, 
Proposition~\ref{theorem:aux} says that if $\operatorname{Im} \gamma(t,x)$ 
is integrable in $x \in \mathbb{R}$ uniformly for $t \in \mathbb{R}$, 
one can solve the initial value problem 
\eqref{equation:pde1}-\eqref{equation:data1}. 
In other words,  
$\operatorname{Im} \gamma(t,x)$ is seemingly an obstruction to the $L^2$-well-posedness, 
but can be resolved by the local smoothing effect of 
$\exp(\sqrt{-1}at\partial^4/\partial x^4)$ described as   
\begin{equation}
\left\lVert
(1+x^2)^{-\delta/4}
\left(-\frac{\partial^2}{\partial x^2}\right)^{3/4} 
\exp\left(\sqrt{-1}at\frac{\partial^4}{\partial x^4}\right)
u_0
\right\rVert_{L^2(\mathbb{R}^2)}
\leqslant
C
\lVert{u_0}\rVert_{L^2(\mathbb{R})},
\quad
\delta>1. 
\label{equation:smoothing}
\end{equation}
This shows that solutions to the initial value problem for $u_t=\sqrt{-1}au_{xxxx}$ gains 
extra smoothness of order $3/2$ in  $x$. 
See \cite{c} for instance. 
We do not need to make full use of \eqref{equation:smoothing} 
for proving Proposition~\ref{theorem:aux}. 
We have only to use the gain of smoothness of order one for this. 
This means that Proposition~\ref{theorem:aux} is never sharp, and 
that the condition given there is too strong. 
In fact, Mizuhara (\cite{mizuhara}) and Tarama (\cite{tarama}) studied 
the necessary and sufficient conditions for $L^2$-well-posedness of 
the initial value problem for more general higher-order linear dispersive 
partial differential equations in one space dimension. 
Under the condition in Proposition~\ref{theorem:aux}, 
one can eliminate $\operatorname{Im} \gamma(t,x)$ in \eqref{equation:pde1} exactly 
by using a gauge transform. 
However, we shall actually make the local smoothing effect visible 
by using another gauge transform with $\phi(x)$. 
This will be convenient for applying our idea to \eqref{equation:pde} 
since the corresponding problematic terms in \eqref{equation:pde} are very complicated.  
\begin{proof}[Proof of Proposition~\ref{theorem:aux}] 
We shall give only the outline of the energy estimates. 
Here we make use of elementary pseudodifferential calculus. 
See \cite{kumano-go} and \cite{taylor} for instance. 
Let $r>0$ be a sufficiently large constant. 
Pick up $\varphi(\xi) \in C^\infty(\mathbb{R})$ such that 
$$
\varphi(\xi)=1 \quad (\lvert\xi\rvert \geqslant r+1),
\qquad
\varphi(\xi)=0 \quad (\lvert\xi\rvert \leqslant r). 
$$
Here we introduce a pseudodifferential operator $\Lambda=I+\tilde{\Lambda}$ of order zero, 
where $I$ is the identity operator, and 
$$
\tilde{\Lambda}v(x)
=
\frac{1}{2\pi}
\iint_{\mathbb{R}^2}
e^{\sqrt{-1}(x-y)\xi}
\tilde{\lambda}(x,\xi)
v(y)
dyd\xi,
$$
$$
\tilde{\lambda}(x,\xi)
=
\Phi(x)\frac{\varphi(\xi)}{4a\xi},
\quad
\Phi(x)
=
\int_0^x
\phi(y)
dy.
$$
$\Lambda$ is a bounded linear operator of $L^2(\mathbb{R})$ to $L^2(\mathbb{R})$, 
and invertible since $\Lambda=I+\mathcal{O}(r^{-1})$ provided that $r$ is sufficiently large. 
Indeed the Neumann series 
$I+\sum_{l=1}^\infty(-\tilde{\Lambda})^l$ 
is the inverse operator of $\Lambda$. 
Hence $\Lambda$ is an automorphism on $L^2(\mathbb{R})$. 
\par
Set $v(t,x)=\Lambda{u}(t,x)$, i.e., $u=\Lambda^{-1}v$. 
Apply $\Lambda$ to \eqref{equation:pde1}. 
Set $D_x=-\sqrt{-1}\partial/\partial x$ for short. 
Here we denote by $\mathscr{L}$ the set of all $L^2$-bounded operators on $\mathbb{R}$. 
In what follows, 
different positive constants are denoted by the same letter $C$, 
and 
different operators in $C(\mathbb{R};\mathscr{L})$ 
are denoted by the same notation $P(t)$.  
Then we have 
\begin{align*}
  \Lambda u_t 
& = v_t, 
\\
  \Lambda \sqrt{-1}au_{xxxx} 
& =
  \sqrt{-1}a(I+\tilde{\Lambda})D_x^4
  (I-\tilde{\Lambda}+\tilde{\Lambda}^2-\tilde{\Lambda}^3+\dotsb)v
\\
& =
  \sqrt{-1}aD_x^4
  (I-\tilde{\Lambda}+\tilde{\Lambda}^2-\tilde{\Lambda}^3+\dotsb)v
\\
& +
  \sqrt{-1}a\tilde{\Lambda}D_x^4
  (I-\tilde{\Lambda}+\tilde{\Lambda}^2-\tilde{\Lambda}^3+\dotsb)v
\\
& =
  \sqrt{-1}aD_x^4v
\\
& -
  \sqrt{-1}aD_x^4\tilde{\Lambda}
  (I-\tilde{\Lambda}+\tilde{\Lambda}^2-\tilde{\Lambda}^3+\dotsb)v
\\
& +
  \sqrt{-1}a\tilde{\Lambda}D_x^4
  (I-\tilde{\Lambda}+\tilde{\Lambda}^2-\tilde{\Lambda}^3+\dotsb)v  
\\
& =
  \sqrt{-1}aD_x^4v
  +
  \sqrt{-1}a\bigl[\tilde{\Lambda},D_x^4\bigr]
  (I-\tilde{\Lambda}+\tilde{\Lambda}^2-\tilde{\Lambda}^3+\dotsb)v
\\
& =
  \sqrt{-1}aD_x^4v
  +
  \sqrt{-1}a\bigl[\tilde{\Lambda},D_x^4\bigr]v
  -
  \sqrt{-1}a\bigl[\tilde{\Lambda},D_x^4\bigr]\tilde{\Lambda}v
  +
  P(t)v,
\\
  \Lambda \sqrt{-1}\{\beta(t,x)u_x\}_x 
& = 
  \sqrt{-1}\{\beta(t,x)v_x\}_x +P(t)v,
\\
  \Lambda\gamma(t,x)u_x
& = 
  \gamma(t,x)v_x+P(t)v.  
\end{align*} 
Since 
$$
\sqrt{-1}a\bigl[\tilde{\Lambda},D_x^4\bigr]
=
\phi(x)\frac{\partial^2}{\partial x^2}
+
\frac{3}{2}\phi^\prime(x)\frac{\partial}{\partial x} 
+
P(t), 
$$
we have 
\begin{align*}
  \Lambda \sqrt{-1}au_{xxxx}
& =
  \sqrt{-1}aD_x^4v
  +
  \phi(x)v_{xx}
  +
  \frac{3}{2}\phi^\prime(x)v_x
  -
  \sqrt{-1}
  \frac{\phi(x)\Phi(x)}{4a}v_x
  +
  P(t)v
\\
& =
  \sqrt{-1}aD_x^4v
  +
  \{\phi(x)v_x\}_x
  +
  \frac{1}{2}\phi^\prime(x)v_x
  -
  \sqrt{-1}
  \frac{\phi(x)\Phi(x)}{4a}v_x
  +
  P(t)v.
\end{align*}
Combining the above, we obtain 
\begin{align*}
  v_t
& =
  \sqrt{-1}av_{xxxx}
  +
  \{\phi(x)v_x\}_x
  +
  \sqrt{-1}\{\beta(t,x)v_x\}_x
\\
& +
  \left\{\operatorname{Re}\gamma(t,x)+\frac{\phi^\prime(x)}{2}\right\}
  v_x
  +
  \sqrt{-1}
  \left\{
  \operatorname{Im}\gamma(t,x)+\frac{\phi(x)\Phi(x)}{4a}
  \right\}
  v_x
  +
  P(t)v.
\end{align*}
Using this, we deduce 
$$
\frac{d}{dt}
\int_{\mathbb{R}}
\lvert{v}\rvert^2
dx
=
-
2
\int_{\mathbb{R}}
\phi(x)
\lvert{v_x}\rvert^2
dx
+
\sqrt{-1}
\int_{\mathbb{R}}
\gamma_1(t,x)
(v_x\overline{v}-v\overline{v_x})
dx
+
2
\operatorname{Re}
\int_{\mathbb{R}}
\{P(t)v\}\overline{v}
dx,
$$
$$
\gamma_1(t,x)
=
\operatorname{Im}\gamma(t,x)+\frac{\phi(x)\Phi(x)}{4a}.
$$
Since $\gamma_1(t,x)=\mathcal{O}\bigl(\phi(x)\bigr)$, 
the Schwarz inequality implies that 
\begin{align*}
  \left\lvert
  \int_{\mathbb{R}}
  \gamma_1(t,x)
  (v_x\overline{v}-v\overline{v_x})
  dx
  \right\rvert
& \leqslant
  C
  \left\{
  \int_{\mathbb{R}}
  \phi(x)
  \lvert{v_x}\rvert^2
  dx
  \right\}^{1/2}
  \left\{
  \int_{\mathbb{R}}
  \phi(x)
  \lvert{v}\rvert^2
  dx
  \right\}^{1/2}
\\
& \leqslant
  C
  \left\{
  \int_{\mathbb{R}}
  \phi(x)
  \lvert{v_x}\rvert^2
  dx
  \right\}^{1/2}
  \left\{
  \int_{\mathbb{R}}
  \lvert{v}\rvert^2
  dx
  \right\}^{1/2}
\\
& \leqslant
  \int_{\mathbb{R}}
  \phi(x)
  \lvert{v_x}\rvert^2
  dx
  +
  C
  \int_{\mathbb{R}}
  \lvert{v}\rvert^2
  dx.
\end{align*}
Hence for any $T>0$ there exists a constant $C_T>0$ depending on $T>0$ such that 
$$
\frac{d}{dt}
\int_{\mathbb{R}}
\lvert{v}\rvert^2
dx
+
\int_{\mathbb{R}}
\phi(x)
\lvert{v_x}\rvert^2
dx
\leqslant
C_T
\int_{\mathbb{R}}
\lvert{v}\rvert^2
dx
$$
for $t\in[0,T]$. This implies that 
$$
\int_{\mathbb{R}}
\lvert{v(t,x)}\rvert^2
dx
+
\int_0^t
e^{C_T(t-s)}
\phi(x)
\lvert{v_x(t,x)}\rvert^2
dx
\leqslant
e^{C_Tt}
\int_{\mathbb{R}}
\lvert{v(0,x)}\rvert^2
dx,
\quad
t\in[0,T].
$$
The same inequality holds for the negative direction of $t$. 
Using these energy estimates, we can prove Proposition~\ref{theorem:aux}. 
We omit the detail. 
\end{proof} 
%
%
\section{Parabolic Regularization}
\label{section:parabolic}
In this section we shall solve 
the initial value problem for a regularized equation of the form 
\begin{alignat}{2}
  u_t
& =
  \bigl(-\varepsilon+aJ(u)\bigr)\nabla_x^3u_x
&
&
\nonumber
\\
& +
  \{1+bg_u(u_x,u_x)\}J(u)\nabla_xu_x
  +
  cg_u(\nabla_xu_x,u_x)J(u)u_x
& \quad
& \text{in}
  \quad
  \mathbb(0,\infty)\times\mathbb{R},
\label{equation:pde2}
\\
  u(0,x)
& =
  u_0(x)
& \quad
& \text{in}
  \quad
  \mathbb{R},
\label{equation:data2}
\end{alignat}
where $\varepsilon$ is a positive parameter. 
We shall use a sequence $\{u^\varepsilon\}_{\varepsilon \in (0,1]}$ 
of the solutions to \eqref{equation:pde2}-\eqref{equation:data2} 
to construct a solution to 
\eqref{equation:pde}-\eqref{equation:data} 
in the next section. 
The results of this section are the following. 
\begin{lemma}
\label{theorem:parabolic1}
Let $k$ be an integer not smaller than four. 
For any initial mapping $u_0 \in H^{k+1}(\mathbb{R};TN)$, 
there exists a positive number $T_\varepsilon$ depending only on 
$\varepsilon>0$ and $\lVert{u_{0x}}\rVert_{H^4(\mathbb{R};TN)}$ 
such that the initial value problem \eqref{equation:pde2}-\eqref{equation:data2} 
has a unique solution $u^\varepsilon \in C\bigl([0,T_\varepsilon];H^{k+1}(\mathbb{R};TN)\bigr)$.   
\end{lemma}
The equation \eqref{equation:pde2} corresponds to 
quasilinear parabolic partial differential equations. 
It is possible to solve \eqref{equation:pde2}-\eqref{equation:data2} directly 
by using the fundamental solution of the differential operator 
$$
\frac{\partial}{\partial t}
-
\bigl(-\varepsilon+aJ(u)\bigr)
\bigl(\nabla^N_{du(\partial/\partial x)}\bigr)^3\frac{\partial}{\partial x}
$$
and so-called the Leray-Schauder fixed point theorem. 
We embed \eqref{equation:pde2} 
in an appropriate Euclidean space $\mathbb{R}^d$ 
by the Nash isometric embedding $w:N \rightarrow \mathbb{R}^d$ 
(See e.g., \cite{GR} and \cite{gunther}), 
and deal with it as a system of quasilinear partial differential equations. 
In particular, we need to take care of the range of solutions to the embedded equation. 
Unfortunately, however, 
the above mentioned approach to \eqref{equation:pde2} seems to be very complicated. 
For this reason, 
we consider more regularized initial value problem of the form 
\begin{alignat}{2}
  u_t
& =
  \delta\nabla_x^5u_x
  +
  \bigl(-\varepsilon+aJ(u)\bigr)\nabla_x^3u_x
&
&
\nonumber
\\
& +
  \{1+bg_u(u_x,u_x)\}J(u)\nabla_xu_x
  +
  cg_u(\nabla_xu_x,u_x)J(u)u_x
& \quad
& \text{in}
  \quad
  \mathbb(0,\infty)\times\mathbb{R},
\label{equation:pde3}
\\
  u(0,x)
& =
  u_0(x)
& \quad
& \text{in}
  \quad
  \mathbb{R},
\label{equation:data3}
\end{alignat}
and we shall prove Lemma~\ref{theorem:parabolic1} elementally. 
Here $\delta$ is a positive parameter.   
Let $\{u^{\varepsilon,\delta}\}_{\delta>0}$ be a sequence of solutions to 
\eqref{equation:pde3}-\eqref{equation:data3} for any fixed $\varepsilon>0$. 
We will get a sequence of solutions $\{u^\varepsilon\}_{\varepsilon>0}$ 
to \eqref{equation:pde2}-\eqref{equation:data2} 
by the standard compactness arguments as $\delta \downarrow 0$. 
We first prove the following. 
\begin{lemma}
\label{theorem:parabolic2}
Let $k$ be an integer not smaller than four. 
For any initial mapping $u_0 \in H^{k+1}(\mathbb{R};TN)$, 
there exists a positive number $T_{\varepsilon,\delta}$ depending only on 
$\varepsilon>0$, $\delta>0$ and $\lVert{u_{0x}}\rVert_{H^4(\mathbb{R};TN)}$ 
such that the initial value problem \eqref{equation:pde3}-\eqref{equation:data3} 
has a unique solution 
$u^{\varepsilon,\delta} \in C\bigl([0,T_{\varepsilon,\delta}];H^{k+1}(\mathbb{R};TN)\bigr)$.   
\end{lemma}
To prove Lemma~\ref{theorem:parabolic2}, 
we embed \eqref{equation:pde3}-\eqref{equation:data3} 
in an appropriate Euclidean space $\mathbb{R}^d$. 
Let $w \in C^\infty(N;\mathbb{R}^d)$ be the Nash isometric embedding, 
and let $w_r(N)$ be a tubular neighborhood of $w(N)$ with sufficiently small $r>0$ 
defined by 
$$
w_r(N)
=
\{
v_0+v_1 \in \mathbb{R}^d 
\ \vert \ 
v_0 \in w(N),\ 
v_1 \in T_{v_0}w(N)^\perp,\ 
\lvert{v_1}\rvert < r\},
$$
where 
$T_{v_0}w(N)^\perp$ is the orthogonal complement of $T_{v_0}w(N)$ 
in $T_{v_0}\mathbb{R}^d\simeq\mathbb{R}^d$. 
The mapping 
$$
\Pi: w_r(N) \ni v_0+v_1 \mapsto v_0 \in w(N)
$$
is the natural projection. 
Set $\rho(v)=v-\Pi(v)$ for $v \in w_r(N)$. 
\par
Let $u$ be a solution to \eqref{equation:pde3}-\eqref{equation:data3}. 
Set $v=w(u)$ for short. Then we have 
\begin{align}
  v_t
& =
  dw(u_t)
  =
  dw
  \Bigl(
  \delta\nabla_x^5u_x
  +
  \bigl(-\varepsilon+aJ(u)\bigr)\nabla_x^3u_x
  +
  \dotsb
  \Bigr)
\nonumber
\\
& =
  \delta v_{xxxxxx}+F(v,v_x,v_{xx},v_{xxx},v_{xxxx},v_{xxxxx}), 
\label{equation:pde4}
\end{align}
where $F \in C^\infty(\mathbb{R}^{6d};\mathbb{R}^d)$ is an appropriate function. 
Set $\tilde{F}(v)=F(v,v_x,v_{xx},v_{xxx},v_{xxxx},v_{xxxxx})$ for short. 
The initial value problem for \eqref{equation:pde4} with 
$v(0,x)=w\bigl(u_0(x)\bigr)$ is equivalent to an integral equation of the form 
\begin{equation}
v(t)
=
\exp\left(t\delta\frac{\partial^6}{\partial x^6}\right)w(u_0)
+
\int_0^t
\exp\left((t-s)\delta\frac{\partial^6}{\partial x^6}\right)
\tilde{F}\bigl(v(s)\bigr)
ds.
\label{equation:inteqn1}
\end{equation}
We shall solve \eqref{equation:inteqn1} without considering the range of $v$, that is, 
we shall deal with 
\begin{equation}
v(t)
=
\exp\left(t\delta\frac{\partial^6}{\partial x^6}\right)w(u_0)
+
\int_0^t
\exp\left((t-s)\delta\frac{\partial^6}{\partial x^6}\right)
\tilde{F}\Bigl(\Pi\bigl(v(s)\bigr)\Bigr)
ds. 
\label{equation:inteqn2}
\end{equation}
If a solution $v$ to \eqref{equation:inteqn2} satisfies $\Pi(v)=v$, 
that is, $\rho(v)=0$, then $v$ solves \eqref{equation:inteqn1} 
and $u=w^{-1}(v)$ is a solution to \eqref{equation:pde3}-\eqref{equation:data3}. 
Note that for any $\alpha>0$, there exists a constant $C_{\delta,\alpha}>0$ 
depending only on $\delta$  and $\alpha$ such that 
$$
\left
\lVert
\exp\left(t\delta\frac{\partial^6}{\partial x^6}\right)
f
\right
\rVert_{H^{6-\alpha}(\mathbb{R})}
\leqslant
\frac{C_{\delta,\alpha}}{t^{(6-\alpha)/6}}
\lVert{f}\rVert_{L^2(\mathbb{R})},
\quad
t>0.
$$
Combining this and the contraction mapping theorem, we can prove the following. 
\begin{lemma}
\label{theorem:parabolic3}
Let $k$ be an integer not smaller than four. 
For any initial mapping $u_0 \in H^{k+1}(\mathbb{R};TN)$, 
there exists a positive number $T_{\varepsilon,\delta}$ depending only on 
$\varepsilon>0$, $\delta>0$ and $\lVert{u_{0x}}\rVert_{H^4(\mathbb{R};TN)}$ 
such that the integral equation \eqref{equation:inteqn2} 
has a unique solution $v$ satisfying 
\begin{equation}
v 
\in 
C
\bigl(
[0,T_{\varepsilon,\delta}]\times\mathbb{R};w_r(N)
\bigr),
\quad
v_x 
\in 
C
\bigl(
[0,T_{\varepsilon,\delta}];
H^k(\mathbb{R};\mathbb{R}^d)
\bigr).
\label{equation:smoothness1}
\end{equation}
\end{lemma}
We omit the detail of the proof of Lemma~\ref{theorem:parabolic3}. 
Here we will complete the proof of Lemma~\ref{theorem:parabolic2}. 
\begin{proof}[Proof of Lemma~\ref{theorem:parabolic2}] 
Let $v$ be a unique solution to \eqref{equation:inteqn2} in Lemma~\ref{theorem:parabolic3}. 
We have only to show that $\rho(v)=v-\Pi(v)=0$. 
Since $\Pi \in \mathscr{B}^\infty(\mathbb{R}^d;\mathbb{R}^d)$, 
the smoothness of $v$ given in \eqref{equation:smoothness1} implies that 
\begin{align}
  \rho(v) 
& \in 
  \mathscr{B}\bigl([0,T_{\varepsilon,\delta}]\times\mathbb{R};\mathbb{R}^d\bigr),
\label{equation:bound201}
\\
  \rho(v)_x
  =
  v_x-\frac{\partial \Pi}{\partial v}(v)v_x 
& \in 
  C\bigl([0,T_{\varepsilon,\delta}];H^4(\mathbb{R};\mathbb{R}^d)\bigr),
\label{equation:bound202}
\\
  \rho(v)_{xx}, \rho(v)_{xxx}
& \in 
  C\bigl([0,T_{\varepsilon,\delta}];H^4(\mathbb{R};\mathbb{R}^d)\bigr),
\label{equation:bound203}
\end{align}
where $\mathscr{B}$ denotes the set of all bounded continuous function. 
We shall evaluate 
$\lvert\rho(v)\rvert^2=\langle\rho(v),\rho(v)\rangle$. 
Fix arbitrary $\eta\in(0,1]$. A simple computation gives 
\begin{equation}
\frac{d}{dt}
\frac{1}{2}
\int_{\mathbb{R}}
e^{-\eta x^2}
\lvert\rho(v)\rvert^2
dx
=
\int_{\mathbb{R}}
e^{-\eta x^2}
\langle\rho(v)_t,\rho(v)\rangle
dx.
\label{equation:201} 
\end{equation}
Since 
$$
\Pi(v)_t
=
\frac{\partial \Pi}{\partial v}(v)v_t 
\in 
T_{\Pi(v)}w(N),
\quad
\rho(v) 
\in 
T_{\Pi(v)}w(N)^\perp,
$$
we have 
$$
\langle\Pi(v)_t,\rho(v)\rangle=0,
\quad
\langle\rho(v)_t,\rho(v)\rangle
=
\Bigl\langle
\bigl\{
\rho(v)_t + \Pi(v)_t 
\bigr\},
\rho(v)
\Bigr\rangle
=
\langle{v_t,\rho(v)}\rangle. 
$$ 
Substituting this into \eqref{equation:201}, we get 
\begin{equation}
\frac{d}{dt}
\frac{1}{2}
\int_{\mathbb{R}}
e^{-\eta x^2}
\lvert\rho(v)\rvert^2
dx
=
\int_{\mathbb{R}}
e^{-\eta x^2}
\langle{v_t,\rho(v)}\rangle
dx.
\label{equation:202} 
\end{equation}
Set $u=w^{-1}\bigl(\Pi(v)\bigr)$ for short. 
Here we remark that 
\begin{align}
  v_t
& =
  \delta v_{xxxxxx}
  +
  \tilde{F}\bigl(\Pi(v)\bigr)
\nonumber
\\
& =
  \delta\rho(v)_{xxxxxx}
  +
  \bigl\{
  \delta\Pi(v)_{xxxxxx}
  +
  \tilde{F}\bigl(\Pi(v)\bigr)
  \bigr\}
\nonumber
\\
& =
  \delta\rho(v)_{xxxxxx}
  +
  dw
  \Bigl(
  \delta\nabla_x^5u_x
  +
  \bigl(-\varepsilon+aJ(u)\bigr)
  \nabla_x^3u_x
  +
  \dotsb
  \Bigr), 
\label{equation:203}
\end{align}
and the second term of the right hand side above belongs to 
$T_{\Pi(v)}w(N)$. 
Substituting \eqref{equation:203} into \eqref{equation:202}, we have 
\begin{equation}
\frac{d}{dt}
\frac{1}{2}
\int_{\mathbb{R}}
e^{-\eta x^2}
\lvert\rho(v)\rvert^2
dx
=
\delta
\int_{\mathbb{R}}
e^{-\eta x^2}
\langle{\rho(v)_{xxxxxx},\rho(v)}\rangle
dx.
\label{equation:204} 
\end{equation}
In view of the integration by parts, 
\eqref{equation:204} becomes 
\begin{equation}
\frac{d}{dt}
\frac{1}{2}
\int_{\mathbb{R}}
e^{-\eta x^2}
\lvert\rho(v)\rvert^2
dx
=
-
\delta
\int_{\mathbb{R}}
\Bigl\langle
\rho(v)_{xxx},
\bigl\{
e^{-\eta x^2} \rho(v)
\bigr\}_{xxx}
\Bigr\rangle
dx.
\label{equation:205} 
\end{equation}
Substituting 
\begin{align*}
  \bigl\{
  e^{-\eta x^2} \rho(v)
  \bigr\}_{xxx} 
& =
  e^{-\eta x^2} 
  \rho(v)_{xxx}
  +
  3
  \bigl\{e^{-\eta x^2}\bigr\}_x 
  \rho(v)_{xx}
\\
& +
  3
  \bigl\{e^{-\eta x^2}\bigr\}_{xx} 
  \rho(v)_x
  +
  \bigl\{e^{-\eta x^2}\bigr\}_{xxx} 
  \rho(v), 
\\
  \bigl\{e^{-\eta x^2}\bigr\}_x 
& =
  -2\eta x e^{-\eta x^2},
\\
  \bigl\{e^{-\eta x^2}\bigr\}_{xx} 
& =
  \bigl(
  -2\eta +4\eta^2 x^2
  \bigr)
  e^{-\eta x^2},
\\
  \bigl\{e^{-\eta x^2}\bigr\}_{xxx} 
& =
  \bigl(
  12\eta^2 x -8\eta^3 x^3
  \bigr)
  e^{-\eta x^2}
\end{align*}
into \eqref{equation:205}, we have 
\begin{align}
  \frac{d}{dt}
  \frac{1}{2}
  \int_{\mathbb{R}}
  e^{-\eta x^2}
  \lvert\rho(v)\rvert^2
  dx
& =
  -
  \delta
  \int_{\mathbb{R}}
  e^{-\eta x^2}
  \lvert\rho(v)_{xxx}\rvert^2
  dx
\nonumber
\\
& +
  \delta
  \int_{\mathbb{R}}
  \bigl(
  -12\eta^2 x + 8\eta^3 x^3
  \bigr)
  e^{-\eta x^2}
  \bigl\langle\rho(v)_{xxx},\rho(v)\bigr\rangle
  dx
\label{equation:206} 
\\
& +
  \delta
  \int_{\mathbb{R}}
  \bigl(
  6\eta - 12\eta^2 x^2
  \bigr)
  e^{-\eta x^2}
  \bigl\langle\rho(v)_{xxx},\rho(v)_x\bigr\rangle
  dx
\label{equation:207} 
\\
& +
  \delta
  \int_{\mathbb{R}}
  6\eta x
  e^{-\eta x^2}
  \bigl\langle\rho(v)_{xxx},\rho(v)_{xx}\bigr\rangle
  dx.
\label{equation:208} 
\end{align}
Here we recall the smoothness of $v$ as in 
\eqref{equation:bound201}, 
\eqref{equation:bound202} 
and 
\eqref{equation:bound203}. 
Combining $\eta \in (0,1]$, 
the Schwarz inequality for the integration on $\mathbb{R}$ 
and the change of variable $y=\eta^{1/2}x$, we deduce that 
\begin{align}
  \lvert\text{\eqref{equation:206}}\rvert
& \leqslant
  12\delta 
  \int_{\mathbb{R}}
  \bigl(
  \eta^2 \lvert{x}\rvert + \eta^3 \lvert{x}\rvert^3
  \bigr)
  e^{-\eta x^2}
  \lvert\rho(v)_{xxx}\rvert
  \lvert\rho(v)\rvert
  dx
\nonumber
\\
& \leqslant
  C\delta\eta^{3/2}
  \int_{\mathbb{R}}
  \bigl(
  \eta^{1/2}\lvert{x}\rvert + \eta^{3/2}\lvert{x}\rvert^3
  \bigr)
  e^{-\eta x^2}
  \lvert\rho(v)_{xxx}\rvert
  dx
\nonumber
\\
& \leqslant
  C\delta\eta^{3/2}
  \int_{\mathbb{R}}
  \bigl(
  1+\eta x^2
  \bigr)^{3/2}
  e^{-\eta x^2}
  \lvert\rho(v)_{xxx}\rvert
  dx
\nonumber
\\
& \leqslant
  C\delta\eta^{3/2}
  \left\{
  \int_{\mathbb{R}}
  (1+\eta x^2)^3
  e^{-2\eta x^2}
  dx
  \right\}^{1/2}
\nonumber
\\
& \leqslant
  C\delta\eta^{5/4}
  \left\{
  \int_{\mathbb{R}}
  (1+y^2)^3
  e^{-2y^2}
  dy
  \right\}^{1/2}
\nonumber
\\
& =
  C\delta\eta^{5/4},
\label{equation:209}  
\\
  \lvert\text{\eqref{equation:207}}\rvert
& \leqslant
  12\delta\eta
  \int_{\mathbb{R}}
  \bigl(
  1+\eta x^2
  \bigr)
  e^{-\eta x^2}
  \lvert\rho(v)_{xxx}\rvert
  \lvert\rho(v)_{x}\rvert
  dx
\nonumber
\\
& \leqslant
  12\delta\eta
  \left\{
  \sup_{y\in\mathbb{R}}
  (1+y^2)e^{-y^2}
  \right\}
  \int_{\mathbb{R}}
  \lvert\rho(v)_{xxx}\rvert
  \lvert\rho(v)_{x}\rvert
  dx
\nonumber
\\
& =
  C\delta\eta
\label{equation:210}  
\\
  \lvert\text{\eqref{equation:208}}\rvert
& \leqslant
  6\delta\eta^{1/2}
  \int_{\mathbb{R}}
  \bigl(
  \eta x^2
  \bigr)^{1/2}
  e^{-\eta x^2}
  \lvert\rho(v)_{xxx}\rvert
  \lvert\rho(v)_{xx}\rvert
  dx
\nonumber
\\
& \leqslant
  6\delta\eta^{1/2}
  \left\{
  \sup_{y\in\mathbb{R}}
  \lvert{y}\rvert e^{-y^2}
  \right\}
  \int_{\mathbb{R}}
  \lvert\rho(v)_{xxx}\rvert
  \lvert\rho(v)_{xx}\rvert
  dx
\nonumber
\\
& =
  C\delta\eta^{1/2},
\label{equation:211} 
\end{align}
where $C$ is a positive constant which is independent of $\eta$. 
Combining 
\eqref{equation:206}, 
\eqref{equation:207},  
\eqref{equation:208}, 
\eqref{equation:209},  
\eqref{equation:210} 
and 
\eqref{equation:211}, 
we deduce that there exists a positive constant $C_0$ which is independent of $\eta$ such that 
$$
\frac{d}{dt} 
\int_{\mathbb{R}} 
e^{-\eta x^2} 
\lvert\rho(v)\rvert^2 
dx
\leqslant 
C_0\delta\eta^{1/2},
\quad
\eta \in (0,1], 
\quad
t \in [0,T_{\varepsilon,\delta}].
$$
Since $\rho\bigl(v(0,\cdot)\bigr)=\rho\bigl(w(u_0)\bigr)=0$, 
we have 
$$
\int_{\mathbb{R}} 
e^{-\eta x^2} 
\bigl\lvert\rho\bigl(v(t,x)\bigr)\bigr\rvert^2 
dx
\leqslant
C_0\delta\eta^{1/2}t,
\quad
t \in [0,T_{\varepsilon,\delta}].
$$
Applying the Beppo-Levi theorem to this, we deduce that 
for any fixed $t \in [0,T_{\varepsilon,\delta}]$ 
$$
0
\leqslant
\int_{\mathbb{R}} 
\bigl\lvert\rho\bigl(v(t,x)\bigr)\bigr\rvert^2 
dx
=
\lim_{\eta \downarrow 0}
\int_{\mathbb{R}} 
e^{-\eta x^2} 
\bigl\lvert\rho\bigl(v(t,x)\bigr)\bigr\rvert^2 
dx
=
0. 
$$
This implies that $\rho\bigl(v(t,\cdot)\bigr)$ belongs to $L^2(\mathbb{R};\mathbb{R}^d)$ 
and  $\rho\bigl(v(t,x)\bigr)=0$ a.e. $x \in \mathbb{R}$ for any 
$t \in [0,T_{\varepsilon,\delta}]$. 
Thus we deduce that 
$\rho\bigl(v(t,x)\bigr)=0$ for any 
$(t,x) \in [0,T_{\varepsilon.\delta}]\times\mathbb{R}$ 
since 
$v \in C\bigl([0,T_{\varepsilon.\delta}]\times\mathbb{R};\mathbb{R}^d\bigr)$.  
\end{proof}
We will conclude the present section with the proof of Lemma~\ref{theorem:parabolic1}. 
\begin{proof}[Proof of Lemma~\ref{theorem:parabolic1}]
Let $\{u^\delta\}_{\delta\in(0,1]}$ be a sequence of solutions to 
\eqref{equation:pde3}-\eqref{equation:data3} with a fixed parameter $\varepsilon$. 
We shall show that there exists $T_\varepsilon$ depending only on 
$\varepsilon$ and $\lVert{u_{0x}}\rVert_{H^4(\mathbb{R};TN)}$ 
such that  $\{u^\delta\}_{\delta\in(0,1]}$ is bounded in the function space 
$ L^\infty\bigl(0,T_\varepsilon;H^{k+1}(\mathbb{R};TN)\bigr)$. 
If this is true, then the standard compact arguments imply that 
there exists a mapping $u$ such that 
$u$ satisfies 
$$
u 
\in 
C\bigl([0,T_\varepsilon];H^k(\mathbb{R};TN)\bigr)
\bigcap 
L^\infty\bigl(0,T_\varepsilon;H^{k+1}(\mathbb{R};TN)\bigr)
$$
and solves 
\eqref{equation:pde2}-\eqref{equation:data2} 
provided that $\delta \downarrow 0$. 
The key of this uniform estimates is also the smoothing property of the parabolic operator 
$\partial/\partial t + \varepsilon\partial^4/\partial x^4$. 
By using this property, we can also prove the uniqueness and the continuity in the time variable, 
but we omit the detail.    
\par
We abbreviate $J(u^\delta)$ and $g_{u^\delta}(\cdot,\cdot)$ by 
$J$ and $g(\cdot,\cdot)$ respectively. 
We will evaluate 
$$
\sum_{l=0}^k
\int_{\mathbb{R}}
g\bigl(\nabla_x^lu^\delta_x,\nabla_x^lu^\delta_x\bigr)
dx.
$$
We apply $\nabla_x^{l+1}$, $l=0,1,2,\dotsc,k$ to 
\begin{align}
  u^\delta_t
& =
  \delta\nabla_x^5u^\delta_x
  +
  \bigl(-\varepsilon+aJ\bigr)\nabla_x^3u^\delta_x
&
&
\nonumber
\\
& +
  \{1+bg(u^\delta_x,u^\delta_x)\}J\nabla_xu^\delta_x
  +
  cg(\nabla_xu^\delta_x,u^\delta_x)Ju^\delta_x
& \quad
& \text{in}
  \quad
  \mathbb(0,T_{\varepsilon,\delta})\times\mathbb{R}. 
\label{equation:pdedelta} 
\end{align}
We compute this term by term. 
Let $R$ be the Riemann curvature tensor of $(N,J,g)$. 
Note that for the left hand side of \eqref{equation:pdedelta}, 
\begin{align}
  \nabla_t\bigl(u^\delta_x\bigr)
& =
  \nabla_xu^\delta_t,
\label{equation:3lhs1}
\\
  \nabla_t
  \bigl(\nabla_x^lu^\delta_x\bigr) 
& =
  \nabla_x^{l+1}u^\delta_t
  +
  \sum_{m=0}^{l-1}
  \nabla_x^{l-1-m}
  \Bigl\{
  R(u^\delta_t,u^\delta_x)\nabla_x^mu^\delta_x
  \Bigr\},
  \quad
  l=1,2,3,\dotsc,
\label{equation:3lhs2}
\end{align}
\begin{align}
& \sum_{m=0}^{l-1}
  \nabla_x^{l-1-m}
  \Bigl\{
  R(u^\delta_t,u^\delta_x)\nabla_x^mu^\delta_x
  \Bigr\}
\nonumber
\\
  =
& \delta
  \sum_{m=0}^{l-1}
  \nabla_x^{l-1-m}
  \Bigl\{
  R\bigl(\nabla_x^5 u^\delta_x,u^\delta_x\bigr)
  \nabla_x^mu^\delta_x
  \Bigr\}
\nonumber
\\
  -
& \varepsilon
  \sum_{m=0}^{l-1}
  \nabla_x^{l-1-m}
  \Bigl\{
  R\bigl(\nabla_x^3 u^\delta_x,u^\delta_x\bigr)
  \nabla_x^mu^\delta_x
  \Bigr\}
\nonumber
\\
  +
& a
  \sum_{m=0}^{l-1}
  \nabla_x^{l-1-m}
  \Bigl\{
  R\bigl(J(u^\delta) \nabla_x^3 u^\delta_x,u^\delta_x\bigr)
  \nabla_x^mu^\delta_x
  \Bigr\}
\nonumber
\\
  +
& \sum_{m=0}^{l-1}
  \nabla_x^{l-1-m}
  \Bigl\{
  \bigl(1+bg(u^\delta_x, u^\delta_x)\bigr)
  R\bigl(J(u^\delta) \nabla_x u^\delta,u^\delta_x\bigr)
  \nabla_x^mu^\delta_x
  \Bigr\}
\nonumber
\\
  +
& c  
  \sum_{m=0}^{l-1}
  \nabla_x^{l-1-m}
  \Bigl\{
  g\bigl(\nabla_x u^\delta_x, u^\delta_x\bigr)
  R\bigl(J(u^\delta) \nabla_x u^\delta,u^\delta_x\bigr)
  \nabla_x^mu^\delta_x
  \Bigr\}
\nonumber
\\
  =
& \delta
  \nabla_x
  \Bigl\{
  R\bigl(\nabla_x^{l+3} u^\delta_x,u^\delta_x\bigr)u^\delta_x
  \Bigr\}
\nonumber
\\
  +
& \delta 
  \mathcal{O}
  \Bigl(
  g\bigl(\nabla_x^{l+3} u^\delta_x, \nabla_x^{l+3} u^\delta_x\bigr)^{1/2}
  \Bigr)
\nonumber
\\
  +
& \mathcal{O}
  \left(
  \sum_{\alpha=1}^2
  g\bigl(\nabla_x^{l+\alpha} u^\delta_x, \nabla_x^{l+\alpha} u^\delta_x\bigr)^{1/2}
  \right)
\nonumber
\\
  +
& \mathcal{O}
  \left(
  \sum_{m=0}^l
  g\bigl(\nabla_x^m u^\delta_x, \nabla_x^m u^\delta_x\bigr)^{1/2}
  \right).
\label{equation:3004}
\end{align}
Here we used the Sobolev embedding. 
On the other hand, we have 
\begin{align}
  \nabla_x^{l+1}
  u^\delta_t
& =
  \delta\nabla_x^6
  \bigl(\nabla_x^lu^\delta_x\bigr)
  +
  \bigl(-\varepsilon+aJ(u^\delta)\bigr)
  \bigl(\nabla_x^lu^\delta_x\bigr)
  +
  J(u^\delta)\nabla_x^2
  \bigl(\nabla_x^lu^\delta_x\bigr)
\nonumber
\\
& +
  b
  \sum_{\mu+\nu=0}^{l+1}
  \frac{(l+1)!}{\mu!\nu!(l+1-\mu-\nu)!}
  g\bigl(\nabla_x^\mu u^\delta_x, \nabla_x^\nu u^\delta_x\bigr)
  J(u^\delta)
  \nabla_x^{l+2-\mu-\nu}u^\delta_x
\nonumber
\\
& +
  c
  \sum_{\mu+\nu=0}^{l+1}
  \frac{(l+1)!}{\mu!\nu!(l+1-\mu-\nu)!}
  g\bigl(\nabla_x^{\mu+1} u^\delta_x, \nabla_x^\nu u^\delta_x\bigr)
  J(u^\delta)
  \nabla_x^{l+1-\mu-\nu}u^\delta_x
\nonumber
\\ 
& =
  \delta\nabla_x^6
  \bigl(\nabla_x^lu^\delta_x\bigr)
  +
  \bigl(-\varepsilon+aJ(u^\delta)\bigr)
  \bigl(\nabla_x^lu^\delta_x\bigr)
  +
  J(u^\delta)\nabla_x^2
  \bigl(\nabla_x^lu^\delta_x\bigr)
  +
  bQ_1+cQ_2.
\label{equation:3001}
\end{align}
We modify  $Q_1$ and $Q_2$ a little for the sake of convenience for our energy estimates: 
\begin{align}
  Q_1
& =
  \sum_{\mu+\nu=0}^{l+1}
  \frac{(l+1)!}{\mu!\nu!(l+1-\mu-\nu)!}
  g\bigl(\nabla_x^\mu u^\delta_x, \nabla_x^\nu u^\delta_x\bigr)
  J(u^\delta)
  \nabla_x^{l+2-\mu-\nu}u^\delta_x
\nonumber
\\
& =
  g(u^\delta_x,u^\delta_x)
  J(u^\delta)
  \nabla_x^2
  \bigl(\nabla_x^lu^\delta_x\bigr)
\nonumber
\\
& +
  2
  g\Bigl(\nabla_x\bigl(\nabla_x^lu^\delta_x\bigr),u^\delta_x\Bigr)
  J(u^\delta)\nabla_xu^\delta_x
\nonumber
\\
& +
  2(l+1)
  g\bigl(\nabla_xu^\delta_x,u^\delta_x\bigr)
  J(u^\delta)\nabla_x\bigl(\nabla_x^lu^\delta\bigr)
\nonumber
\\
& +
  \sum_{\substack{\mu+\nu=2 \\ \mu, \nu \leqslant l}}^{l+1}
  \frac{(l+1)!}{\mu!\nu!(l+1-\mu-\nu)!}
  g\bigl(\nabla_x^\mu u^\delta_x, \nabla_x^\nu u^\delta_x\bigr)
  J(u^\delta)
  \nabla_x^{l+2-\mu-\nu}u^\delta_x
\nonumber    
\\
& =
  \nabla_x
  \Bigl\{
  g(u^\delta_x,u^\delta_x)
  J(u^\delta)
  \nabla_x
  \bigl(\nabla_x^lu^\delta_x\bigr)
  \Bigr\}
\nonumber
\\
& +
  2
  g\Bigl(\nabla_x\bigl(\nabla_x^lu^\delta_x\bigr),u^\delta_x\Bigr)
  J(u^\delta)\nabla_xu^\delta_x
\nonumber
\\
& +
  2l
  g\bigl(\nabla_xu^\delta_x,u^\delta_x\bigr)
  J(u^\delta)\nabla_x\bigl(\nabla_x^lu^\delta\bigr)
\nonumber
\\
& +
  \sum_{\substack{\mu+\nu=2 \\ \mu, \nu \leqslant l}}^{l+1}
  \frac{(l+1)!}{\mu!\nu!(l+1-\mu-\nu)!}
  g\bigl(\nabla_x^\mu u^\delta_x, \nabla_x^\nu u^\delta_x\bigr)
  J(u^\delta)
  \nabla_x^{l+2-\mu-\nu}u^\delta_x,
\label{equation:3002}
\\
  Q_2
& =
  \sum_{\mu+\nu=0}^{l+1}
  \frac{(l+1)!}{\mu!\nu!(l+1-\mu-\nu)!}
  g\bigl(\nabla_x^{\mu+1} u^\delta_x, \nabla_x^\nu u^\delta_x\bigr)
  J(u^\delta)
  \nabla_x^{l+1-\mu-\nu}u^\delta_x
\nonumber
\\
& =
  g
  \bigl(\nabla_x u^\delta_x, u^\delta_x\bigr)
  J(u^\delta)
  \nabla_x
  \bigl(\nabla_x^lu^\delta_x\bigr)
\nonumber
\\
& +
  g\Bigl(\nabla_x^2\bigl(\nabla_x^lu^\delta_x\bigr), u^\delta_x\Bigr)
  J(u^\delta)
  u^\delta_x
\nonumber
\\
& +
  (l+1)
  g\Bigl(\nabla_x\bigl(\nabla_x^lu^\delta_x\bigr), \nabla_xu^\delta_x\Bigr)
  J(u^\delta)
  u^\delta_x
\nonumber
\\
& +
  (l+1)
  g\Bigl(\nabla_x\bigl(\nabla_x^lu^\delta_x\bigr), u^\delta_x\Bigr)
  J(u^\delta)
  \nabla_xu^\delta_x
\nonumber
\\
& +
  g\Bigl(\nabla_xu^\delta_x, \nabla_x\bigl(\nabla_x^lu^\delta_x\bigr)\Bigr)
  J(u^\delta)
  u^\delta_x
\nonumber
\\
& +
  \sum_{\substack{\mu+\nu+\rho=l+1 \\ \mu \leqslant l-1, \ \nu, \rho \leqslant l}}
  \frac{(l+1)!}{\mu!\nu!\rho!}
  g\bigl(\nabla_x^{\mu+1} u^\delta_x,\nabla_x^\nu u^\delta_x\bigr)
  J(u^\delta)
  \nabla_x^\rho 
  u^\delta_x
\nonumber
\\
& =
  \nabla_x
  \Bigl\{
  g\Bigl(\nabla_x\bigl(\nabla_x^l u^\delta_x\bigr), u^\delta_x\Bigr)
  J(u^\delta)
  u^\delta_x
  \Bigr\}
\nonumber
\\
& +
  g
  \bigl(\nabla_x u^\delta_x, u^\delta_x\bigr)
  J(u^\delta)
  \nabla_x
  \bigl(\nabla_x^lu^\delta_x\bigr)
\nonumber
\\
& +
  (l+1)
  g\Bigl(\nabla_x\bigl(\nabla_x^lu^\delta_x\bigr), \nabla_xu^\delta_x\Bigr)
  J(u^\delta)
  u^\delta_x
\nonumber
\\
& +
  l
  g\Bigl(\nabla_x\bigl(\nabla_x^lu^\delta_x\bigr), u^\delta_x\Bigr)
  J(u^\delta)
  \nabla_xu^\delta_x
\nonumber
\\
& +
  \sum_{\substack{\mu+\nu+\rho=l+1 \\ \mu \leqslant l-1, \ \nu, \rho \leqslant l}}
  \frac{(l+1)!}{\mu!\nu!\rho!}
  g\bigl(\nabla_x^{\mu+1} u^\delta_x,\nabla_x^\nu u^\delta_x\bigr)
  J(u^\delta)
  \nabla_x^\rho 
  u^\delta_x.
\label{equation:3003}
\end{align}
Combining 
\eqref{equation:3lhs1}, 
\eqref{equation:3lhs2}, 
\eqref{equation:3004}, 
\eqref{equation:3001}, 
\eqref{equation:3002} 
and  
\eqref{equation:3003}, 
we have 
\begin{align}
  \nabla_t
  \bigl(\nabla_x^l u^\delta_x\bigr)
& =
  \delta
  \nabla_x^6
  \bigl(\nabla_x^l u^\delta_x\bigr)
  +
  \bigl(-\varepsilon+J(u^\delta)\bigr)
  \nabla_x^4
  \bigl(\nabla_x^l u^\delta_x\bigr)
\nonumber
\\
& +
  \delta
  \nabla_x
  \Bigl\{
  R\bigl(\nabla_x^{l+3} u^\delta_x,u^\delta_x\bigr)u^\delta_x
  \Bigr\}
  +
  \delta
  \mathcal{O}
  \Bigl(
  g\bigl(\nabla_x^{l+3} u^\delta_x, \nabla_x^{l+3} u^\delta_x\bigr)^{1/2}
  \Bigr)
\nonumber
\\
& +
  \mathcal{O}
  \left(
  \sum_{\alpha=1}^2
  g\bigl(\nabla_x^{l+\alpha} u^\delta_x, \nabla_x^{l+\alpha} u^\delta_x\bigr)^{1/2}
  \right)
  +
  \mathcal{O}
  \left(
  \sum_{m=0}^l
  g\bigl(\nabla_x^m u^\delta_x, \nabla_x^m u^\delta_x\bigr)^{1/2}
  \right). 
\label{equation:3005}
\end{align}
Now we compute 
$$
\frac{d}{dt}
\frac{1}{2}
\sum_{l=0}^k
\int_{\mathbb{R}}
g\bigl(\nabla_x^l u^\delta_x, \nabla_x^l u^\delta_x\bigr)
dx
=
\sum_{l=0}^k
\int_{\mathbb{R}}
g\bigl(\nabla_t \nabla_x^l u^\delta_x, \nabla_x^l u^\delta_x\bigr)
dx. 
$$
Substitute \eqref{equation:3005} into this. 
Using the integration by parts and the property of $J$, we have 
\begin{align}
& \frac{d}{dt}
  \frac{1}{2}
  \sum_{l=0}^k
  \int_{\mathbb{R}}
  g\bigl(\nabla_x^l u^\delta_x, \nabla_x^l u^\delta_x\bigr)
  dx
\nonumber
\\
  =
& -
  \delta
  \sum_{l=0}^k
  \int_{\mathbb{R}}
  g
  \Bigl(
  \nabla_x^3 \bigl(\nabla_x^l u^\delta_x\bigr), 
  \nabla_x^3 \bigl(\nabla_x^l u^\delta_x\bigr)
  \Bigr)
  dx
\nonumber
\\
& -
  \varepsilon
  \sum_{l=0}^k
  \int_{\mathbb{R}}
  g
  \Bigl(
  \nabla_x^2 \bigl(\nabla_x^l u^\delta_x\bigr), 
  \nabla_x^2 \bigl(\nabla_x^l u^\delta_x\bigr)
  \Bigr)
  dx
\nonumber
\\
& -
  \delta
  \sum_{l=0}^k
  \int_{\mathbb{R}}
  g
  \Bigl(
  R\bigl(\nabla_x^{l+3} u^\delta_x,u^\delta_x\bigr)u^\delta_x, 
  \nabla_x \bigl(\nabla_x^l u^\delta_x\bigr)
  \Bigr)
  dx
\nonumber
\\
& +
  \delta
  \sum_{l=0}^k
  \int_{\mathbb{R}}
  \mathcal{O}
  \Bigl(
  g\bigl(\nabla_x^{l+3} u^\delta_x, \nabla_x^{l+3} u^\delta_x\bigr)^{1/2}
  g\bigl(\nabla_x^l u^\delta_x, \nabla_x^l u^\delta_x\bigr)^{1/2}
  \Bigr)
  dx
\nonumber
\\
& +
  \sum_{l=0}^k
  \int_{\mathbb{R}}
  \mathcal{O}
  \Bigl(
  \sum_{\alpha=1}^2
  g\bigl(\nabla_x^{l+\alpha} u^\delta_x, \nabla_x^{l+\alpha} u^\delta_x\bigr)^{1/2}
  g\bigl(\nabla_x^l u^\delta_x, \nabla_x^l u^\delta_x\bigr)^{1/2}
  \Bigr)
  dx
\nonumber
\\
& +
  \sum_{l=0}^k
  \int_{\mathbb{R}}
  \mathcal{O}
  \Bigl(
  g\bigl(\nabla_x^l u^\delta_x, \nabla_x^l u^\delta_x\bigr)
  \Bigr)
  dx
\nonumber
\\
  =
& -
  \delta
  \sum_{l=0}^k
  \int_{\mathbb{R}}
  g
  \Bigl(
  \nabla_x^3 \bigl(\nabla_x^l u^\delta_x\bigr), 
  \nabla_x^3 \bigl(\nabla_x^l u^\delta_x\bigr)
  \Bigr)
  dx
\nonumber
\\
& -
  \varepsilon
  \sum_{l=0}^k
  \int_{\mathbb{R}}
  g
  \Bigl(
  \nabla_x^2 \bigl(\nabla_x^l u^\delta_x\bigr), 
  \nabla_x^2 \bigl(\nabla_x^l u^\delta_x\bigr)
  \Bigr)
  dx
\nonumber
\\
& +
  \delta
  \sum_{l=0}^k
  \int_{\mathbb{R}}
  \mathcal{O}
  \Bigl(
  \sum_{\beta=0}^1
  g\bigl(\nabla_x^{l+3} u^\delta_x, \nabla_x^{l+3} u^\delta_x\bigr)^{1/2}
  g\bigl(\nabla_x^{l+\beta} u^\delta_x, \nabla_x^{l+\beta} u^\delta_x\bigr)^{1/2}
  \Bigr)
  dx
\nonumber
\\
& +
  \sum_{l=0}^k
  \int_{\mathbb{R}}
  \mathcal{O}
  \Bigl(
  \sum_{\alpha=1}^2
  g\bigl(\nabla_x^{l+\alpha} u^\delta_x, \nabla_x^{l+\alpha} u^\delta_x\bigr)^{1/2}
  g\bigl(\nabla_x^l u^\delta_x, \nabla_x^l u^\delta_x\bigr)^{1/2}
  \Bigr)
  dx
\nonumber
\\
& +
  \sum_{l=0}^k
  \int_{\mathbb{R}}
  \mathcal{O}
  \Bigl(
  g\bigl(\nabla_x^l u^\delta_x, \nabla_x^l u^\delta_x\bigr)
  \Bigr)
  dx.
\label{equation:3006}
\end{align} 
Integration by parts, 
the Schwarz inequality 
and 
an elementary inequality $2ab \leqslant a^2+b^2$ for $a,b>0$  
give  
\begin{align*}
& \sum_{l=0}^k
  \int_{\mathbb{R}}
  g\bigl(\nabla_x^{l+1} u^\delta_x, \nabla_x^{l+1} u^\delta_x\bigr)
  dx
\\
  =
& -
  \sum_{l=0}^k
  \int_{\mathbb{R}}
  g\bigl(\nabla_x^{l+2} u^\delta_x, \nabla_x^l u^\delta_x\bigr)
  dx
\\
  \leqslant
& \sum_{l=0}^k
  \int_{\mathbb{R}}
  g\bigl(\nabla_x^{l+2} u^\delta_x, \nabla_x^{l+2} u^\delta_x\bigr)^{1/2}
  g\bigl(\nabla_x^l u^\delta_x, \nabla_x^l u^\delta_x\bigr)^{1/2}
  dx
\\
  \leqslant
& \frac{\varepsilon}{2}
  \sum_{l=0}^k
  \int_{\mathbb{R}}
  g\bigl(\nabla_x^{l+2} u^\delta_x, \nabla_x^{l+2} u^\delta_x\bigr)
  dx
  +
  \frac{1}{2\varepsilon}
  \sum_{l=0}^k
  \int_{\mathbb{R}}
  g\bigl(\nabla_x^l u^\delta_x, \nabla_x^l u^\delta_x\bigr)
  dx,
\end{align*}
\begin{align*}
& \sum_{l=0}^k
  \int_{\mathbb{R}}
  \mathcal{O}
  \Bigl(
  g\bigl(\nabla_x^{l+2} u^\delta_x, \nabla_x^{l+2} u^\delta_x\bigr)^{1/2}
  g\bigl(\nabla_x^l u^\delta_x, \nabla_x^l u^\delta_x\bigr)^{1/2}
  \Bigr)
  dx
\\
  \leqslant
& C
  \sum_{l=0}^k
  \int_{\mathbb{R}}
  \Bigl(
  g\bigl(\nabla_x^{l+2} u^\delta_x, \nabla_x^{l+2} u^\delta_x\bigr)^{1/2}
  g\bigl(\nabla_x^l u^\delta_x, \nabla_x^l u^\delta_x\bigr)^{1/2}
  dx
\\
  \leqslant
&  \frac{\varepsilon}{2}
  \sum_{l=0}^k
  \int_{\mathbb{R}}
  g\bigl(\nabla_x^{l+2} u^\delta_x, \nabla_x^{l+2} u^\delta_x\bigr)
  dx
  +
  \frac{C^2}{2\varepsilon}
  \sum_{l=0}^k
  \int_{\mathbb{R}}
  g\bigl(\nabla_x^l u^\delta_x, \nabla_x^l u^\delta_x\bigr)
  dx.
\end{align*}
Using these and an elementary inequality $2ab \leqslant a^2+b^2$ for $a,b>0$ again,  
we obtain 
\begin{align}
& \left\lvert
  \delta
  \sum_{l=0}^k
  \int_{\mathbb{R}}
  \mathcal{O}
  \Bigl(
  \sum_{\beta=0}^1
  g\bigl(\nabla_x^{l+3} u^\delta_x, \nabla_x^{l+3} u^\delta_x\bigr)^{1/2}
  g\bigl(\nabla_x^{l+\beta} u^\delta_x, \nabla_x^{l+\beta} u^\delta_x\bigr)^{1/2}
  \Bigr)
  dx
  \right\rvert 
\nonumber
\\
  \leqslant
& \frac{\delta}{2}
  \sum_{l=0}^k
  \int_{\mathbb{R}}
  g\bigl(\nabla_x^{l+3} u^\delta_x, \nabla_x^{l+3} u^\delta_x\bigr)
  dx
\nonumber
\\
  +
& C\delta
  \sum_{l=0}^k
  \int_{\mathbb{R}}
  g\bigl(\nabla_x^{l+1} u^\delta_x, \nabla_x^{l+1} u^\delta_x\bigr)
  dx
  +
  C\delta
  \sum_{l=0}^k
  \int_{\mathbb{R}}
  g\bigl(\nabla_x^l u^\delta_x, \nabla_x^l u^\delta_x\bigr)
  dx
\nonumber
\\
  \leqslant
& \frac{\delta}{2}
  \sum_{l=0}^k
  \int_{\mathbb{R}}
  g\bigl(\nabla_x^{l+3} u^\delta_x, \nabla_x^{l+3} u^\delta_x\bigr)
  dx
  +
  \frac{\varepsilon}{2}
  \sum_{l=0}^k
  \int_{\mathbb{R}}
  g\bigl(\nabla_x^{l+2} u^\delta_x, \nabla_x^{l+2} u^\delta_x\bigr)
  dx
\nonumber
\\
  +
& \left(
  \frac{C^2\delta^2}{2\varepsilon}
  +
  C\delta
  \right)
  \sum_{l=0}^k
  \int_{\mathbb{R}}
  g\bigl(\nabla_x^l u^\delta_x, \nabla_x^l u^\delta_x\bigr)
  dx,
\label{equation:3007} 
\\
& \left\lvert
  \sum_{l=0}^k
  \int_{\mathbb{R}}
  \mathcal{O}
  \Bigl(
  \sum_{\alpha=1}^2
  g\bigl(\nabla_x^{l+\alpha} u^\delta_x, \nabla_x^{l+\alpha} u^\delta_x\bigr)^{1/2}
  g\bigl(\nabla_x^l u^\delta_x, \nabla_x^l u^\delta_x\bigr)^{1/2}
  \Bigr)
  dx
  \right\rvert
\nonumber
\\
  \leqslant
& C
  \sum_{l=0}^k
  \int_{\mathbb{R}}
  \sum_{\alpha=1}^2
  g\bigl(\nabla_x^{l+\alpha} u^\delta_x, \nabla_x^{l+\alpha} u^\delta_x\bigr)^{1/2}
  g\bigl(\nabla_x^l u^\delta_x, \nabla_x^l u^\delta_x\bigr)^{1/2}
  dx
\nonumber
\\
  =
& C
  \sum_{l=0}^k
  \int_{\mathbb{R}}
  g\bigl(\nabla_x^{l+2} u^\delta_x, \nabla_x^{l+2} u^\delta_x\bigr)^{1/2}
  g\bigl(\nabla_x^l u^\delta_x, \nabla_x^l u^\delta_x\bigr)^{1/2}
  dx
\nonumber
\\
  +
& C
  \sum_{l=0}^k
  \int_{\mathbb{R}}
  g\bigl(\nabla_x^{l+1} u^\delta_x, \nabla_x^{l+1} u^\delta_x\bigr)^{1/2}
  g\bigl(\nabla_x^l u^\delta_x, \nabla_x^l u^\delta_x\bigr)^{1/2}
  dx
\nonumber
\\
  \leqslant
& \left(
  \frac{\varepsilon}{4}
  +
  \frac{\varepsilon}{4}
  \right)
  \sum_{l=0}^k
  \int_{\mathbb{R}}
  g\bigl(\nabla_x^{l+2} u^\delta_x, \nabla_x^{l+2} u^\delta_x\bigr)
  dx
\nonumber
\\
  +
& \left(
  \frac{C^2}{\varepsilon}
  +
  \frac{C^2}{2\varepsilon}
  \right)
   \sum_{l=0}^k
  \int_{\mathbb{R}}
  g\bigl(\nabla_x^{l} u^\delta_x, \nabla_x^{l} u^\delta_x\bigr)
  dx.
\label{equation:3008} 
\end{align} 
Substitute 
\eqref{equation:3007} 
and 
\eqref{equation:3008} 
into 
\eqref{equation:3006}. 
We deduce that there exists a positive constant $C(\varepsilon)$ 
depending only on $\varepsilon\in(0,1]$ and 
$\lVert{u_{0x}}\rVert_{H^4(\mathbb{R};TN)}$ 
such that 
$$
\frac{1}{2}
\frac{d}{dt}
\sum_{l=0}^k
\int_{\mathbb{R}}
g\bigl(\nabla_x^l u^\delta_x, \nabla_x^l u^\delta_x\bigr)
dx
\leqslant
C(\varepsilon)
\int_{\mathbb{R}}
g\bigl(\nabla_x^l u^\delta_x, \nabla_x^l u^\delta_x\bigr)
dx,
\quad
t\in[0,T_{\varepsilon,\delta}].
$$
This implies that there exists $T_\varepsilon>0$ 
depending only on $\varepsilon\in(0,1]$ and 
$\lVert{u_{0x}}\rVert_{H^4(\mathbb{R};TN)}$ 
such that 
$\{u^\delta\}_{\delta\in(0,1]}$ is bounded in 
$L^\infty(0,T_\varepsilon;H^k(\mathbb{R};TN))$. 
The standard compactness arguments shows the existence of solution to 
\eqref{equation:pde2}-\eqref{equation:data2}. 
The uniqueness and the continuity in time of solutions can be proved by the same energy method. 
We omit the detail. 
\end{proof}
\section{Uniform Energy Estimates}
\label{section:energy}
In this section we shall obtain uniform energy estimates of 
$\{u^\varepsilon\}_{\varepsilon\in(0,1]}$, and construct a time-local solution to 
\eqref{equation:pde1}-\eqref{equation:data1} 
by the standard compactness argument. 
More precisely we shall show that there exists $T>0$ which is independent of 
$\varepsilon$ such that $T_\varepsilon \geqslant T$ and that 
$\{u_\varepsilon\}_{\varepsilon\in(0,1]}$ 
is bounded in 
$L^\infty\bigl(0,T;H^{k+1}(\mathbb{R};TN)\bigr)$. 
For this purpose, we need to overcome the loss of one derivative 
and introduce a gauge transform of sections of the pullback bundle 
$(u^\varepsilon)^{-1}TN$ of the form 
\begin{align}
  V^\varepsilon_k
& =
  \nabla_x^ku^\varepsilon_x
  +
  \frac{M}{4a}
  \Phi^\varepsilon(t,x)
  J(u^\varepsilon)
  \nabla_x^{k-1}u^\varepsilon_x,
\label{equation:gauge1}
\\
  \Phi^\varepsilon(t,x)
& =
  \int_{-\infty}^x
  g\bigl(u^\varepsilon_y(t,y),u^\varepsilon_y(t,y)\bigr)
  dy,
\nonumber  
\end{align}
where $M$ is a positive constant determined later. 
The second term of the right hand side of \eqref{equation:gauge1} 
corresponds to the pseudodifferential operator $\tilde{\Lambda}$ 
introduced in Section~\ref{section:mizuhara}. 
We will obtain the uniform bounds of 
$$
\mathscr{N}(u^\varepsilon)^2
=
\int_{\mathbb{R}}
\left\{
g(V^\varepsilon_k,V^\varepsilon_k)
+
\sum_{l=0}^{k-1}
g\bigl(\nabla_x^lu^\varepsilon_x,\nabla_x^lu^\varepsilon_x\bigr)
\right\}
dx. 
$$
In view of the Sobolev embedding with $k\geqslant4$, 
it is easy to see that 
$\{\mathscr{N}(u^\varepsilon)\}_{\varepsilon\in(0,1]}$ 
is bounded in $L^\infty(0,T_\varepsilon)$ 
if and only if 
$\{u^\varepsilon\}_{\varepsilon\in(0,1]}$ 
is bounded in $L^\infty(0,T_\varepsilon;H^k(\mathbb{R};TN))$.  
Hence we shall show that there exists $T>0$ which is independent of 
$\varepsilon\in(0,1]$ such that 
$T_\varepsilon \geqslant T$ and that 
$\{\mathscr{N}(u^\varepsilon)\}_{\varepsilon\in(0,1]}$ 
is bounded in $L^\infty(0,T_\varepsilon)$. 
\begin{proof}[{\bf Proof of Existence}] 
Let $k\geqslant4$. 
Firstly we compute the energy estimates for 
$\nabla_x^lu^\varepsilon_x$ with $l=0,1,\dotsc,k-1$. 
In the same way as the previous section we have for $l=0,1,\dotsc,k$ 
\begin{align}
  \nabla_t\nabla_x^lu^\varepsilon_x
& =
  \bigl(-\varepsilon+aJ(u^\varepsilon)\bigr)
  \nabla_x^4
  \bigl(\nabla_x^lu^\varepsilon_x\bigr)
  +
  J(u^\varepsilon)
  \nabla_x^2
  \bigl(\nabla_x^lu^\varepsilon_x\bigr)
\nonumber
\\
& +
  \nabla_x
  \Bigl\{
  R\Bigl(
   \bigl(-\varepsilon+aJ(u^\varepsilon)\bigr)\nabla_x\bigl(\nabla_x^lu^\varepsilon_x\bigr),u_x
   \Bigr)u^\varepsilon_x
\nonumber
\\
& \qquad\quad
  +
  b
  g(u^\varepsilon_x,u^\varepsilon_x)
  J(u^\varepsilon)
  \nabla_x\bigl(\nabla_x^lu^\varepsilon_x\bigr)
\nonumber
\\
& \qquad\quad
  +
  c
  g\Bigl(\nabla_x\bigl(\nabla_x^lu^\varepsilon_x\bigr),u^\varepsilon_x\Bigr)
  J(u^\varepsilon)
  u^\varepsilon_x
  \Bigr\}
\nonumber
\\
& +
  \mathcal{O}
  \biggl(
  g\Bigl(\nabla_x\bigl(\nabla_x^lu^\varepsilon_x\bigr),\nabla_x\bigl(\nabla_x^lu^\varepsilon_x\bigr)\Bigr)^{1/2}
  g\bigl(\nabla_xu^\varepsilon_x,\nabla_xu^\varepsilon_x\bigr)^{1/2}
  g(u^\varepsilon_x,u^\varepsilon_x)^{1/2}
  \biggr)
\nonumber
\\
& +
  \mathcal{O}
  \left(
  \sum_{j=0}^l
  g\bigl(\nabla_x^ju^\varepsilon_x,\nabla_x^ju^\varepsilon_x\bigr)^{1/2}
  \right).
\label{equation:pde10}
\end{align}
Applying the integration by parts to 
the second term of the right hand side of the above for $l=0,1,\dotsc,k-1$, 
we deduce that 
\begin{equation}
\frac{d}{dt}
\sum_{l=0}^{k-1}
\int_{\mathbb{R}}
g\bigl(\nabla_x^lu^\varepsilon_x,\nabla_x^lu^\varepsilon_x\bigr) 
dx
\leqslant
-
2\varepsilon
\sum_{l=0}^{k-1}
\int_{\mathbb{R}}
g\bigl(\nabla_x^{l+2}u^\varepsilon_x,\nabla_x^{l+2}u^\varepsilon_x\bigr) 
dx
+
C_1
\mathscr{N}(u^\varepsilon)^2, 
\label{equation:energy1}
\end{equation}
where $C_1$ is a positive constant which is independent of $\varepsilon\in(0,1]$. 
\par
Secondly we consider the energy estimates for $V^\varepsilon_k$. 
For this purpose we shall obtain the partial differential equation satisfied by $V_k$. 
The principal part of $V^\varepsilon_k$ is $\nabla_x^ku^\varepsilon_x$ 
and satisfies \eqref{equation:pde10} with $l=k$. 
We shall obtain the equation for the lower order term of $V^\varepsilon_k$. 
We begin with 
\begin{equation}
\nabla_t
\left(
\frac{M}{4a}
\Phi^\varepsilon(t,x)
\nabla_x^{k-1}
u^\varepsilon_x
\right)
=
\frac{M}{4a}
\Phi^\varepsilon(t,x)
\nabla_t
\nabla_x^{k-1}
u^\varepsilon_x
+
\left\{
\frac{M}{2a}
\int_{-\infty}^x
g\bigl(\nabla_tu^\varepsilon_x,u^\varepsilon_x\bigr)
dy
\right\}
\nabla_x^{k-1}
u^\varepsilon_x.
\label{equation:pde11}
\end{equation}
Substitute \eqref{equation:pde10} into the right hand side of \eqref{equation:pde11}. 
The first term of it becomes 
\begin{align}
  \frac{M}{4a}
  \Phi^\varepsilon(t,x)
  \nabla_t
  \nabla_x^{k-1}
  u^\varepsilon_x
& =
  \frac{M}{4a}
  \Phi^\varepsilon(t,x)
\nonumber
\\
& \times
  \Biggl[
  \bigl(-\varepsilon+aJ(u^\varepsilon)\bigr)
  \nabla_x^4
  \bigl(\nabla_x^{k-1}u^\varepsilon_x\bigr)
  +
  J(u^\varepsilon)
  \nabla_x^2
  \bigl(\nabla_x^{k-1}u^\varepsilon_x\bigr)
\nonumber
\\
& \qquad
  +
  \nabla_x
  \Bigl\{
  R\Bigl(
   \bigl(-\varepsilon+aJ(u^\varepsilon)\nabla_x^ku^\varepsilon_x,u_x\bigr)
   \Bigr)u^\varepsilon_x
\nonumber
\\
& \qquad\qquad\quad
  +
  b
  g(u^\varepsilon_x,u^\varepsilon_x)
  J(u^\varepsilon)
  \nabla_x^ku^\varepsilon_x
\nonumber
\\
& \qquad\qquad\quad
  +
  c
  g\Bigl(\nabla_x^ku^\varepsilon_x,u^\varepsilon_x\Bigr)
  J(u^\varepsilon)
  u^\varepsilon_x
  \Bigr\}   
\nonumber
\\
& \qquad
  +
  \mathcal{O}
  \left(
  \sum_{l=0}^k
  g\bigl(\nabla_x^lu^\varepsilon_x,\nabla_x^lu^\varepsilon_x\bigr)^{1/2}
  \right)  
  \Biggr]
\nonumber
\\
& =
  \bigl(-\varepsilon+aJ(u^\varepsilon)\bigr)
  \nabla_x^4
  \left\{
  \frac{M}{4a}
  \Phi^\varepsilon(t,x)
  \nabla_x^{k-1}
  u^\varepsilon_x
  \right\}
\nonumber
\\
& +
  J(u^\varepsilon)
  \nabla_x^2
  \left\{
  \frac{M}{4a}
  \Phi^\varepsilon(t,x)
  \nabla_x^{k-1}
  u^\varepsilon_x
  \right\}
\nonumber
\\
& -
  \bigl(-\varepsilon+aJ(u^\varepsilon)\bigr)
  \sum_{j=1}^4
  \frac{4!}{j!(4-j)!}
  \left\{
  \frac{\partial^j}{\partial x^j} 
  \Phi^\varepsilon(t,x)
  \right\}
  \frac{M}{4a}
  \nabla_x^{k+3-j}
  u^\varepsilon_x
\nonumber
\\
& -
  J(u^\varepsilon)
  \sum_{j=1}^2
  \frac{2!}{j!(2-j)!}
  \left\{
  \frac{\partial^j}{\partial x^j} 
  \Phi^\varepsilon(t,x)
  \right\}
  \frac{M}{4a}
  \nabla_x^{k+1-j}
  u^\varepsilon_x
\nonumber
\\
& +
  \mathcal{O}
  \biggl(
  g\bigl(\nabla_x^{k+1}u^\varepsilon_x,\nabla_x^{k+1}u^\varepsilon_x\bigr)^{1/2}
  g(u^\varepsilon_x,u^\varepsilon_x)^{1/2}
  \biggr)
\nonumber
\\
& +
  \mathcal{O}
  \left(
  \sum_{l=0}^k
  g\bigl(\nabla_x^lu^\varepsilon_x,\nabla_x^lu^\varepsilon_x\bigr)^{1/2}
  \right).
\label{equation:lower101}  
\end{align}
Since $\{\Phi^\varepsilon\}_x=g(u^\varepsilon_x,u^\varepsilon_x)$, 
\begin{align*}
  \sum_{j=1}^4
  \frac{4!}{j!(4-j)!}
  \left\{
  \frac{\partial^j}{\partial x^j} 
  \Phi^\varepsilon(t,x)
  \right\}
  \frac{M}{4a}
  \nabla_x^{k+3-j}
  u^\varepsilon_x
& =
  \frac{M}{a}
  \nabla_x
  \Bigl\{
  g(u^\varepsilon_x,u^\varepsilon_x)
  \nabla_x^{k+1}u^\varepsilon_x
  \Bigr\}
\\
& +
  \mathcal{O}
  \biggl(
  g\bigl(\nabla_x^{k+1}u^\varepsilon_x,\nabla_x^{k+1}u^\varepsilon_x\bigr)^{1/2}
  g(u^\varepsilon_x,u^\varepsilon_x)^{1/2}
  \biggr)
\nonumber
\\
& +
  \mathcal{O}
  \left(
  \sum_{l=0}^k
  g\bigl(\nabla_x^lu^\varepsilon_x,\nabla_x^lu^\varepsilon_x\bigr)^{1/2}
  \right). 
\end{align*}
Substituting this into \eqref{equation:lower101}, we have 
\begin{align}
  \frac{M}{4a}
  \Phi^\varepsilon(t,x)
  \nabla_t
  \nabla_x^{k-1}
  u^\varepsilon_x
& =
  \bigl(-\varepsilon+aJ(u^\varepsilon)\bigr)
  \nabla_x^4
  \left\{
  \frac{M}{4a}
  \Phi^\varepsilon(t,x)
  \nabla_x^{k-1}
  u^\varepsilon_x
  \right\}
\nonumber
\\
& +
  J(u^\varepsilon)
  \nabla_x^2
  \left\{
  \frac{M}{4a}
  \Phi^\varepsilon(t,x)
  \nabla_x^{k-1}
  u^\varepsilon_x
  \right\}
\nonumber
\\
& -
  M
  \left(
  -\frac{\varepsilon}{a}+J(u^\varepsilon)
  \right)
  \nabla_x
  \Bigl\{
  g(u^\varepsilon_x,u^\varepsilon_x)
  \nabla_x^{k+1}u^\varepsilon_x
  \Bigr\}
\nonumber
\\
& +
  \mathcal{O}
  \biggl(
  g\bigl(\nabla_x^{k+1}u^\varepsilon_x,\nabla_x^{k+1}u^\varepsilon_x\bigr)^{1/2}
  g(u^\varepsilon_x,u^\varepsilon_x)^{1/2}
  \biggr)
\nonumber
\\
& +
  \mathcal{O}
  \left(
  \sum_{l=0}^k
  g\bigl(\nabla_x^lu^\varepsilon_x,\nabla_x^lu^\varepsilon_x\bigr)^{1/2}
  \right).
\label{equation:lower102}   
\end{align}
On the other hand, 
\begin{align}
  \left\{
  \frac{M}{2a}
  \int_{-\infty}^x
  g\bigl(\nabla_tu^\varepsilon_x,u^\varepsilon_x\bigr)
  dy
  \right\}
  \nabla_x^{k-1} 
& =
  \left\{
  \frac{M}{2a}
  \int_{-\infty}^x
  g\bigl(-\varepsilon\nabla_x^4u^\varepsilon_x+\dotsb,u^\varepsilon_x\bigr)
  dy
  \right\}
  \nabla_x^{k-1}
  \nonumber
\\
& = 
  \mathcal{O}
  \Bigl(
  \lVert{u^\varepsilon_x}\rVert_{H^4}^2
  \Bigr)
  \nabla_x^{k-1}.
\label{equation:lower103}
\end{align}
Multiply \eqref{equation:pde11} by $J(u^\varepsilon)$, 
and substitute \eqref{equation:lower102} and \eqref{equation:lower103} into it. 
We deduce 
\begin{align}
  \nabla_t
  \left(
  \frac{M}{4a}
  \Phi^\varepsilon(t,x)
  J(u^\varepsilon)
  \nabla_x^{k-1}
  u^\varepsilon_x
  \right)
& =
  \bigl(-\varepsilon+aJ(u^\varepsilon)\bigr)
  \nabla_x^4
  \left\{
  \frac{M}{4a}
  \Phi^\varepsilon(t,x)
  J(u^\varepsilon)
  \nabla_x^{k-1}
  u^\varepsilon_x
  \right\}
\nonumber
\\
& +
  J(u^\varepsilon)
  \nabla_x^2
  \left\{
  \frac{M}{4a}
  \Phi^\varepsilon(t,x)
  J(u^\varepsilon)
  \nabla_x^{k-1}
  u^\varepsilon_x
  \right\}
\nonumber
\\
& +
  M
  \left(
  1+\frac{\varepsilon}{a}J(u^\varepsilon)
  \right)
  \nabla_x
  \Bigl\{
  g(u^\varepsilon_x,u^\varepsilon_x)
  \nabla_x^{k+1}u^\varepsilon_x
  \Bigr\}
\nonumber
\\
& +
  \mathcal{O}
  \biggl(
  g\bigl(\nabla_x^{k+1}u^\varepsilon_x,\nabla_x^{k+1}u^\varepsilon_x\bigr)^{1/2}
  g(u^\varepsilon_x,u^\varepsilon_x)^{1/2}
  \biggr)
\nonumber
\\
& +
  \mathcal{O}
  \left(
  \sum_{l=0}^k
  g\bigl(\nabla_x^lu^\varepsilon_x,\nabla_x^lu^\varepsilon_x\bigr)^{1/2}
  \right).
\label{equation:lower104}    
\end{align}
Combining \eqref{equation:pde10} with $l=k$ and \eqref{equation:lower104}, 
we obtain 
\begin{align}
  \nabla_tV_k
& =
  \bigl(-\varepsilon+aJ(u^\varepsilon)\bigr)
  \nabla_x^4
  V_k^\varepsilon
  +
  J(u^\varepsilon)
  \nabla_x^2
  V_k^\varepsilon
  +
  M
  \left(
  1+\frac{\varepsilon}{a}J(u^\varepsilon)
  \right)
  \nabla_x
  \Bigl\{
  g(u^\varepsilon_x,u^\varepsilon_x)
  \nabla_x
  V_k^\varepsilon
  \Bigr\}
\nonumber
\\
& +
  \nabla_x
  \Bigl\{
  R\Bigl(
   \bigl(-\varepsilon+aJ(u^\varepsilon)\bigr)\nabla_x\bigl(\nabla_x^ku^\varepsilon_x\bigr),u_x
   \Bigr)u^\varepsilon_x
\nonumber
\\
& \qquad\quad
  +
  b
  g(u^\varepsilon_x,u^\varepsilon_x)
  J(u^\varepsilon)
  \nabla_x\bigl(\nabla_x^ku^\varepsilon_x\bigr)
  +
  c
  g\Bigl(\nabla_x\bigl(\nabla_x^ku^\varepsilon_x\bigr),u^\varepsilon_x\Bigr)
  J(u^\varepsilon)
  u^\varepsilon_x
  \Bigr\}
\nonumber
\\ 
& +
  \mathcal{O}
  \biggl(
  g\bigl(\nabla_x^{k+1}u^\varepsilon_x,\nabla_x^{k+1}u^\varepsilon_x\bigr)^{1/2}
  g(u^\varepsilon_x,u^\varepsilon_x)^{1/2}
  \biggr)
  +
  \mathcal{O}
  \left(
  \sum_{l=0}^k
  g\bigl(\nabla_x^lu^\varepsilon_x,\nabla_x^lu^\varepsilon_x\bigr)^{1/2}
  \right).
\label{equation:pde12}    
\end{align}
By using \eqref{equation:pde12} and integration by parts, we deduce 
that there exists positive constants $C_2$ and $C_3$ 
which are independent of $\varepsilon\in(0,1]$ 
such that 
\begin{align}
  \frac{d}{dt}
  \int_{\mathbb{R}}
  g(V_k^\varepsilon,V_k^\varepsilon)
  dx
  =
& 2
  \int_{\mathbb{R}}
  g\bigl(\nabla_tV_k^\varepsilon,V_k^\varepsilon\bigr)
  dx
\nonumber
\\
  \leqslant
& -
  2\varepsilon 
  \int_{\mathbb{R}}
  g\bigl(\nabla_x^2V_k^\varepsilon,\nabla_x^2V_k^\varepsilon\bigr)
  dx
\nonumber
\\
& -
  (2M-C_2) 
  \int_{\mathbb{R}}
  g(u^\varepsilon_x,u^\varepsilon_x)
  g\bigl(\nabla_xV_k^\varepsilon,\nabla_xV_k^\varepsilon\bigr)
  dx
  +
  C_3
  \mathscr{N}(u^\varepsilon)^2.
\label{equation:energy2}
\end{align}
Combining \eqref{equation:energy1} and \eqref{equation:energy2}, 
we deduce that there exists a positive constant $C_4$
which is independent of $\varepsilon\in(0,1]$ 
such that 
$$
\frac{d}{dt}
\mathscr{N}(u^\varepsilon)
\leqslant
C_4
\mathscr{N}(u^\varepsilon),
\quad
t\in[0,T_\varepsilon], 
$$
and 
$$
\mathscr{N}\bigl(u^\varepsilon(t)\bigr)
\leqslant
\mathscr{N}(u_0)
e^{C_4t},
\quad
t\in[0,T_\varepsilon]
$$
provided that $M$ is sufficiently large. 
This shows that there exists a positive time $T$ such that 
$T \leqslant T_\varepsilon$ and that 
$\{u^\varepsilon\}_{\varepsilon\in(0,1]}$ 
is bounded in 
$L^\infty(0,T;H^{k+1}(\mathbb{R};TN))$. 
Thus the standard compactness arguments imply that 
there exists $u \in L^\infty(0,T;H^{k+1}(\mathbb{R};TN))$ solving 
\eqref{equation:pde}-\eqref{equation:data}. 
Moreover, the lower semicontinuity of the norm shows that 
\begin{equation}
\mathscr{N}\bigl(u(t)\bigr)
\leqslant
\mathscr{N}(u_0)
e^{C_4t},
\quad
t\in[0,T_\varepsilon],
\label{equation:energy3} 
\end{equation}
which will be used in the next section.  
\end{proof}
%
%
\section{Uniqueness and Continuity in Time} 
\label{section:recovery}
Finally in this section we prove the uniqueness of solution 
and recover the continuity of the unique solution in time. 
\begin{proof}[{\bf Proof of Uniqueness}]
Let 
$u,v \in C\bigl([0,T]\times\mathbb{R};N\bigr) \cap L^\infty\bigl(0,T;H^7(\mathbb{R};TN)\bigr)$ 
be solutions to 
\eqref{equation:pde1}-\eqref{equation:data1}. 
Set 
\begin{alignat*}{3}
  U
& =
  w(u),
& \quad
  \tilde{U}
& =
  dw(\tilde{u}),
& \quad
  \tilde{u}
& =
  \nabla_xu_x
  +
  \frac{M}{4a}
  \Phi(t,x)
  J(u)
  u_x,
\\
  V
& =
  w(v),
& \quad
  \tilde{V}
& =
  dw(\tilde{v}),
& \quad
  \tilde{v}
& =
  \nabla_xv_x
  +
  \frac{M}{4a}
  \Phi(t,x)
  J(v)
  v_x,
\end{alignat*}
\begin{align*}
  \phi(t,x)
& =
  \sum_{l=0}^2
  \Bigl\{
  g_u\bigl(\nabla_x^lu_x(t,x),\nabla_x^lu_x(t,x)\bigr)
  +
  g_v\bigl(\nabla_x^lv_x(t,x),\nabla_x^lv_x(t,x)\bigr)
  \Bigr\},
\\
  \Phi(t,x) 
& =
  \int_{-\infty}^x
  \phi(t,y)
  dy.
\end{align*}
It suffices to show that 
$U(t,x)=V(t,x)$ on $[0,T]\times\mathbb{R}$ 
for sufficiently small $T>0$. 
For this reason, we may assume that 
$U(t,x)$ and $V(t,x)$ are in the same local coordinate patch on $w(N)$ 
for each $(t,x)\in[0,T]\times\mathbb{R}$. 
Note that 
$$
U_t=dw(u_t),
\quad
U_x=dw(u_x),
\quad
V_t=dw(v_t),
\quad
V_x=dw(v_x).
$$
We shall obtain the partial differential equations satisfied by 
$Z=U-V$, $Z_x=U_x-V_x$ and $\tilde{Z}=\tilde{U}-\tilde{V}$, 
and evaluate 
$$
D(t)^2
=
\frac{1}{2}
\int_{\mathbb{R}}
\bigl\{
\lvert{Z(t,x)}\rvert^2
+
\lvert{Z_x(t,x)}\rvert^2
+
\lvert{\tilde{Z}(t,x)}\rvert^2
\bigr\}
dx. 
$$
Note that $D(0)=0$. 
In the same way as the uniform energy estimates in the previous section, 
the highest order derivative $\tilde{Z}$ gains extra smoothness to obtain the energy estimates. 
In what follows different positive constants are denoted by the same letter $C$. 
\par
First we compute the image of \eqref{equation:pde} by $dw$ for the estimate of $Z$ and $Z_x$. 
Let $\nu_j(U)$, ($j=2n+1,\dotsb,d$) be local expression of an orthonormal basis of 
$T_Uw(N)^\perp$. 
Let $P(U):T_U\mathbb{R}^d \rightarrow T_Uw(N)$ be the orthogonal projection. 
Set $\tilde{J}(U)\xi=dw\Bigl(J\bigl(w^{-1}(U))dw^{-1}(\xi)\Bigr)$ 
for $\xi \in T_Uw(N)$. 
We shall obtain the expression of $dw\bigl(J(u)\nabla_x^3u_x\bigr)$. 
Since $\nabla_xJ(u)=0$ and $dw(\nabla_x\dotsb)=P(U)\{dw(\dotsb)\}_x$, we deduce that 
\begin{align}
  dw\bigl(J(u)\nabla_x^3u_x\bigr)
& =
  dw\Bigl(\nabla_x^2\bigl(J(u)\nabla_xu_x\bigr)\Bigr)
\nonumber
\\
& =
  P(U)
  \Bigl[
  P(U)
  \bigl\{
  \tilde{J}(U)P(U)U_{xx}
  \bigr\}_x
  \Bigr]_x
\nonumber
\\
& =
  \Bigl[
  P(U)
  \bigl\{
  \tilde{J}(U)P(U)U_{xx}
  \bigr\}_x
  \Bigr]_x
\nonumber
\\
& -
  \sum_{j=2n+1}^d
  \Bigl\langle
  \Bigl[
  P(U)
  \bigl\{
  \tilde{J}(U)P(U)U_{xx}
  \bigr\}_x
  \Bigr]_x,
  \nu_j(U)
  \Bigr\rangle
  \nu_j(U). 
\label{equation:401} 
\end{align}
Since 
$\langle{P(U)\dotsb,\nu_j(U)}\rangle=0$, $j=2n+1,\dotsc,d$, 
we have 
$$
\Bigl\langle
\Bigl[
P(U)
\bigl\{
\tilde{J}(U)P(U)U_{xx}
\bigr\}_x
\Bigr]_x,
\nu_j(U)
\Bigr\rangle
=
-
\left\langle
P(U)
\bigl\{
\tilde{J}(U)P(U)U_{xx}
\bigr\}_x,
\frac{\partial \nu_j}{\partial U}(U)
U_x
\right\rangle.
$$
Then, the equality \eqref{equation:401} becomes 
\begin{align}
  dw\bigl(J(u)\nabla_x^3u_x\bigr)
& =
  \Bigl[
  P(U)
  \bigl\{
  \tilde{J}(U)P(U)U_{xx}
  \bigr\}_x
  \Bigr]_x
\nonumber
\\
& +
  \sum_{j=2n+1}^d
  \left\langle
  P(U)
  \bigl\{
  \tilde{J}(U)P(U)U_{xx}
  \bigr\}_x,
  \frac{\partial \nu_j}{\partial U}(U)
  U_x
  \right\rangle 
  \nu_j(U).
\label{equation:402}
\end{align}
In the same way, we obtain 
\begin{equation}
dw\bigl(J(u)\nabla_x^3u_x\bigr)
=
dw\Bigl(\nabla_x\bigl(J(u)\nabla_x^2u_x\bigr)\Bigr)
=
P(U)
\Bigl[
\tilde{J}(U)
P(U)
\bigl\{
P(U)U_{xx}
\bigr\}_x
\Bigr]_x.
\label{equation:403} 
\end{equation}
It is relatively easy to compute the image of the second and the third terms 
of the right hand side of \eqref{equation:pde} by $dw$. 
Thus we obtain 
\begin{align}
  U_t
& =
  a
  \Bigl[
  P(U)
  \bigl\{
  \tilde{J}(U)P(U)U_{xx}
  \bigr\}_x
  \Bigr]_x
\nonumber
\\
& +
  a
  \sum_{j=2n+1}^d
  \left\langle
  P(U)
  \bigl\{
  \tilde{J}(U)P(U)U_{xx}
  \bigr\}_x,
  \frac{\partial \nu_j}{\partial U}(U)
  U_x
  \right\rangle 
  \nu_j(U)
\nonumber
\\
& +
  \bigl\{
  1+b\langle{U_x,U_x}\rangle
  \bigr\}
  \tilde{J}(U)P(U)U_{xx}
  +
  c
  \bigl\langle
  P(U)U_{xx},U_x
  \bigr\rangle
  \tilde{J}(U)U_x
\label{equation:404}
\\
& =
  a
  P(U)
  \Bigl[
  \tilde{J}(U)
  P(U)
  \bigl\{
  P(U)U_{xx}
  \bigr\}_x
  \Bigr]_x
\nonumber
\\
& +
  \bigl\{
  1+b\langle{U_x,U_x}\rangle
  \bigr\}
  \tilde{J}(U)P(U)U_{xx}
  +
  c
  \bigl\langle
  P(U)U_{xx},U_x
  \bigr\rangle
  \tilde{J}(U)U_x.
\label{equation:405} 
\end{align}
Using \eqref{equation:404}, we have 
\begin{align}
  Z_t
& =
  a
  \Bigl[
  P(U)
  \bigl\{
  \tilde{J}(U)P(U)Z_{xx}
  \bigr\}_x
  \Bigr]_x 
\nonumber
\\
& +
  a
  \Bigl[
  P(U)
  \bigl\{
  \tilde{J}(U)P(U)V_{xx}
  \bigr\}_x
  -
  P(V)
  \bigl\{
  \tilde{J}(V)P(V)V_{xx}
  \bigr\}_x
  \Bigr]_x 
\nonumber
\\
& +
  a
  \sum_{j=2n+1}^d
  \left\langle
  P(U)
  \bigl\{
  \tilde{J}(U)P(U)Z_{xx}
  \bigr\}_x,
  \frac{\partial \nu_j}{\partial U}(U)
  U_x
  \right\rangle
  \nu_j(U)
\nonumber
\\
& +
  a
  \sum_{j=2n+1}^d
  \bigg[
  \left\langle
  P(U)
  \bigl\{
  \tilde{J}(U)P(U)V_{xx}
  \bigr\}_x,
  \frac{\partial \nu_j}{\partial U}(U)
  U_x
  \right\rangle
  \nu_j(U)
\nonumber
\\
& \qquad\qquad\qquad
  -
  \left\langle
  P(V)
  \bigl\{
  \tilde{J}(V)P(V)V_{xx}
  \bigr\}_x,
  \frac{\partial \nu_j}{\partial V}(V)
  V_x
  \right\rangle
  \nu_j(V)
  \bigg]
\nonumber
\\
& +
  \mathcal{O}
  \bigl(
  \lvert{Z_{xx}}\rvert
  +
  \lvert{Z_x}\rvert
  +
  \lvert{Z}\rvert
  \bigr)
\nonumber
\\
& =
  a
  \Bigl[
  P(U)
  \bigl\{
  \tilde{J}(U)P(U)Z_{xx}
  \bigr\}_x
  \Bigr]_x 
  +
  \mathcal{O}
  \bigl(
  \lvert{U_x}\rvert \lvert{Z_{xxx}}\rvert
  +
  \lvert{Z_{xx}}\rvert
  +
  \lvert{Z_x}\rvert
  +
  \lvert{Z}\rvert
  \bigr). 
\label{equation:406}
\end{align}
Similarly, using \eqref{equation:405}, we have 
\begin{equation}
Z_t
=
a
P(U)
\Bigl[
\tilde{J}(U)
P(U)
\bigl\{
P(U)Z_{xx}
\bigr\}_x
\Bigr]_x 
+
\mathcal{O}
\bigl(
\lvert{Z_{xx}}\rvert
+
\lvert{Z_x}\rvert
+
\lvert{Z}\rvert
\bigr).  
\label{equation:407}
\end{equation}
It follows that 
$Z \in C^1\bigl([0,T];L^2(\mathbb{R};\mathbb{R}^d)\bigr)$ 
from 
$Z_t \in C\bigl([0,T];L^2(\mathbb{R};\mathbb{R}^d)\bigr)$ 
and 
$Z(0)=0 \in L^2(\mathbb{R};\mathbb{R}^d)$. 
In particular, 
$Z(t) \in L^2(\mathbb{R};\mathbb{R}^d)$ for all $t\in[0,T]$ is guaranteed. 
By using \eqref{equation:406} and integration by parts, we deduce 
\begin{align}
  \frac{d}{dt}
  \frac{1}{2}
  \int_{\mathbb{R}}
  \lvert{Z}\rvert^2
  dx
& =
  \int_{\mathbb{R}}
  \langle{Z_t,Z}\rangle
  dx
\nonumber
\\
& =
  a
  \int_{\mathbb{R}}
  \Bigl\langle
  \Bigl[
  P(U)
  \bigl\{
  \tilde{J}(U)P(U)Z_{xx}
  \bigr\}_x
  \Bigr]_x,
  Z
  \Bigr\rangle
  dx
\nonumber
\\
& +
  \int_{\mathbb{R}}
  \Bigl\langle
  \mathcal{O}\bigl(\lvert{Z_{xxx}}\rvert+\lvert{Z_{xx}}\rvert+\lvert{Z_x}\rvert+\lvert{Z}\rvert\bigr),
  Z
  \Bigr\rangle
  dx
\nonumber
\\
& =
  -a
  \int_{\mathbb{R}}
  \Bigl\langle
  P(U)
  \bigl\{
  \tilde{J}(U)P(U)Z_{xx}
  \bigr\}_x,
  Z_x
  \Bigr\rangle
  dx
\nonumber
\\
& +
  \mathcal{O}
  \left(
  \int_{\mathbb{R}}
  \bigl(
  \lvert{Z_{xx}}\rvert^2+\lvert{Z_x}\rvert^2+\lvert{Z}\rvert^2
  \bigr)
  dx
  \right)
\nonumber
\\
& =
  a
  \int_{\mathbb{R}}
  \Bigl\langle
  \tilde{J}(U)P(U)Z_{xx},
  \bigl\{P(U)Z_x\bigr\}_x
  \Bigr\rangle
  dx
\nonumber
\\
& +
  \mathcal{O}
  \left(
  \int_{\mathbb{R}}
  \bigl(
  \lvert{Z_{xx}}\rvert^2+\lvert{Z_x}\rvert^2+\lvert{Z}\rvert^2
  \bigr)
  dx
  \right)
\nonumber
\\
& \leqslant
  C D(t)^2.
\label{equation:408}
\end{align}
By using \eqref{equation:407}, we deduce 
\begin{align}
  \frac{d}{dt}
  \frac{1}{2}
  \int_{\mathbb{R}}
  \lvert{Z_x}\rvert^2
  dx
& =
  \int_{\mathbb{R}}
  \langle{Z_{xt}},Z_x\rangle
  dx
\nonumber
\\
& =
  -
  \int_{\mathbb{R}}
  \langle{Z_t,Z_{xx}}\rangle
  dx
\nonumber
\\
& =
  -
  a
  \int_{\mathbb{R}}
  \Bigl\langle
  P(U)
  \Bigl[
  \tilde{J}(U)P(U)
  \bigl\{
  P(U)Z_{xx}
  \bigr\}_x
  \Bigr]_x,
  Z_{xx}
  \Bigr\rangle
  dx
  +
  \mathcal{O}\bigl(D(t)^2\bigr)
\nonumber
\\
& =
  a
  \int_{\mathbb{R}}
  \Bigl\langle
  \tilde{J}(U)P(U)
  \bigl\{
  P(U)Z_{xx}
  \bigr\}_x,
  \bigl\{
  P(U)Z_{xx}
  \bigr\}_x
  \Bigr\rangle
  dx
  +
  \mathcal{O}\bigl(D(t)^2\bigr)
\nonumber
\\
& =
  \mathcal{O}\bigl(D(t)^2\bigr)
  \leqslant
  C D(t)^2.
\label{equation:409}
\end{align}
\par
Next we shall obtain the equation for $\tilde{Z}$ and evaluate it. 
The equation for $\tilde{u}$ is 
\begin{align}
  \nabla_t\tilde{u}
& =
  aJ(u)\nabla_x^4\tilde{u}
  +
  J(u)\nabla_x^2\tilde{u}
  +
  M\nabla_x\bigl\{\phi(t,x)\nabla_x\tilde{u}\bigr\}
\nonumber
\\
& +
  \nabla_x
  \Bigl[
  R\bigl(J(u)\nabla_x\tilde{u},u_x\bigr)u_x
  +
  b
  g_u(u_x,u_x)J(u)\nabla_x\tilde{u}
  +
  c
  g_u(\nabla_x\tilde{u},u_x)J(u)u_x
  \Bigr]
\nonumber
\\
& +
  \mathcal{O}
  \bigl(
  \lvert{u_x}\rvert \lvert{\nabla_x\tilde{u}}\rvert \lvert{\tilde{u}}\rvert
  +
  \lvert\tilde{u}\rvert+\lvert{u_x}\rvert+\lvert{u}\rvert
  \bigr).  
\label{equation:450}
\end{align}
Set $\xi(u)$ as the right hand side of \eqref{equation:pde} for short. 
We have 
$$
dw\bigl(\xi(u)\bigr)
=
\mathcal{O}
\Bigl(
\lvert\tilde{U}_{xx}\rvert
+
\lvert\tilde{U}_x\rvert
+
\lvert\tilde{U}\rvert
+
\lvert{U_x}\rvert
\Bigr). 
$$
We compute the image of \eqref{equation:450} by $dw$. 
Since $\bigl\langle{\tilde{U},\nu_j(U)}\bigr\rangle=0$, 
the left hand side of \eqref{equation:450} becomes 
\begin{align}
  dw(\nabla_t\tilde{u})
& =
  \tilde{U}_t
  -
  \sum_{j=2n+1}^d
  \bigl\langle
  \tilde{U}_t,\nu_j(U)
  \bigr\rangle 
  \nu_j(U)
\nonumber
\\
& =
  \tilde{U}_t
  +
  \sum_{j=2n+1}^d
  \left\langle
  \tilde{U},\frac{\partial \nu_j}{\partial U}(U)U_t
  \right\rangle 
  \nu_j(U)
\nonumber
\\
& =
  \tilde{U}_t
  +
  \sum_{j=2n+1}^d
  \left\langle
  \tilde{U},
  \frac{\partial \nu_j}{\partial U}(U)
  dw\bigl(\xi(u)\bigr)
  \right\rangle 
  \nu_j(U)
\nonumber
\\
& =
  \tilde{U}_t
  +
  \sum_{j=2n+1}^d
  \Bigl\langle
  \tilde{U},
  \mathcal{O}
  \Bigl(
  \lvert\tilde{U}_{xx}\rvert
  +
  \lvert\tilde{U}_x\rvert
  +
  \lvert\tilde{U}\rvert
  +
  \lvert{U_x}\rvert
  \Bigr)
  \Bigr\rangle 
  \nu_j(U).
\label{equation:451}
\end{align} 
In the same way as \eqref{equation:405} 
together with \eqref{equation:450} and \eqref{equation:451}, 
we deduce that 
\begin{align*}
  \tilde{U}_t
& =
  a
  P(U)
  \Bigl[
  P(U)
  \bigl\{
  \tilde{J}(U)
  P(U)
  \bigl(
  P(U)
  \tilde{U}_x
  \bigr)_x
  \bigr\}_x
  \Bigr]_x
\\
\nonumber
& +
  P(U)
  \bigl\{
  \tilde{J}(U)
  P(U)
  \tilde{U}_x
  \bigr\}_x
  +
  M
  \bigl(
  \phi(t,x)
  \tilde{U}_x
  \bigr)_x
\\
& +
  \sum_{j=2n+1}^d
  \Bigl\langle
  \tilde{U},
  \mathcal{O}
  \Bigl(
  \lvert\tilde{U}_{xx}\rvert
  +
  \lvert\tilde{U}_x\rvert
  +
  \lvert\tilde{U}\rvert
  +
  \lvert{U_x}\rvert
  \Bigr)
  \Bigr\rangle 
  \nu_j(U)
\\
& +
  \mathcal{O}
  \Bigl(
  \bigl(
  \phi(t,x)
  \tilde{U}_x
  \bigr)_x  
  \Bigr)
  +
  \mathcal{O}
  \Bigl(
  \phi(t,x)^{1/2}
  \lvert\tilde{U}_x\rvert
  +
  \lvert{\tilde{U}}\rvert
  +
  \lvert{U_x}\rvert
  \Bigr).
\end{align*}
Taking the difference between the equations for $\tilde{U}$ and $\tilde{V}$, we obtain 
\begin{align*}
  \tilde{Z}_t
& =
  a
  P(U)
  \Bigl[
  P(U)
  \bigl\{
  \tilde{J}(U)
  P(U)
  \bigl(
  P(U)
  \tilde{Z}_x
  \bigr)_x
  \bigr\}_x
  \Bigr]_x
\\
\nonumber
& +
  P(U)
  \bigl\{
  \tilde{J}(U)
  P(U)
  \tilde{Z}_x
  \bigr\}_x
  +
  M
  \bigl(
  \phi(t,x)
  \tilde{Z}_x
  \bigr)_x
\\ 
& +
  \sum_{j=2n+1}^d
  \Bigl\langle
  \tilde{U},
  \mathcal{O}
  \Bigl(
  \lvert\tilde{Z}_{xx}\rvert
  \Bigr)
  \Bigr\rangle 
  \nu_j(U)
\\
& +
  \mathcal{O}
  \Bigl(
  \bigl(
  \phi(t,x)
  \tilde{Z}_x
  \bigr)_x  
  \Bigr)
  +
  \mathcal{O}
  \Bigl(
  \phi(t,x)^{1/2}
  \lvert\tilde{Z}_x\rvert
  +
  \lvert{\tilde{Z}}\rvert
  +
  \lvert{Z_x}\rvert
  +
  \lvert{Z}\rvert
  \Bigr).
\end{align*}
Using this and integration by parts, 
we deduce that there exists a positive constant $C_1$ such that 
\begin{align}
  \frac{d}{dt}
  \frac{1}{2}
  \int_{\mathbb{R}}
  \langle{\tilde{Z},\tilde{Z}}\rangle
  dx
& =
  \int_{\mathbb{R}}
  \langle{\tilde{Z}_t,\tilde{Z}}\rangle
  dx
\nonumber
\\
& \leqslant
  a
  \int_{\mathbb{R}}
  \Bigl\langle
  P(U)
  \Bigl[
  P(U)
  \bigl\{
  \tilde{J}(U)
  P(U)
  \bigl(
  P(U)
  \tilde{Z}_x
  \bigr)_x
  \bigr\}_x
  \Bigr]_x,
  \tilde{Z}
  \Bigr\rangle
  dx
\label{equation:410}
\\
& +
  \int_{\mathbb{R}}
  \Bigl\langle
  P(U)
  \bigl\{
  \tilde{J}(U)
  P(U)
  \tilde{Z}_x
  \bigr\}_x,
  \tilde{Z}
  \Bigr\rangle
  dx
\label{equation:411}
\\
& +
  \sum_{j=2n+1}^d
  \int_{\mathbb{R}}
  \mathcal{O}
  \Bigl(
  \phi(t,x)^{1/2}
  \lvert\tilde{Z}_{xx}\rvert
  \Bigr)
  \langle
  \nu_j(U),
  \tilde{Z}
  \rangle
  dx
\label{equation:412}
\\
& -
  (M-C_1)
  \int_{\mathbb{R}}
  \phi(t,x)
  \lvert{\tilde{Z}_x}\rvert^2
  dx
  +
  C D(t)^2.
\nonumber  
\end{align}
We will evaluate 
\eqref{equation:410}, 
\eqref{equation:411} 
and 
\eqref{equation:412} 
respectively. 
First we remark that 
\begin{equation}
\bigl\langle
\nu_j(U),\tilde{Z}
\bigr\rangle
=
\bigl\langle
\nu_j(U),\tilde{U}-\tilde{V}
\bigr\rangle 
=
-
\bigl\langle
\nu_j(U),\tilde{V}
\bigr\rangle 
=
-
\bigl\langle
\nu_j(U)-\nu_j(V),\tilde{V}
\bigr\rangle 
=
\mathcal{O}
\Bigl(
\phi(t,x)^{1/2}
Z
\Bigr)
\label{equation:452}
\end{equation}
since 
$\bigl\langle{\nu_j(U),\tilde{U}}\bigr\rangle=0$ 
and 
$\bigl\langle{\nu_j(V),\tilde{V}}\bigr\rangle=0$. 
Applying \eqref{equation:452} and integration by parts to \eqref{equation:412}, we have 
\begin{equation}
\text{\eqref{equation:412}}
\leqslant
C D(t)^2. 
\label{equation:413}
\end{equation}
Here we note that 
\begin{align}
& \Bigl\langle
  \Bigl[
  P(U)
  \bigl\{
  \tilde{J}(U)
  P(U)
  \bigl(
  P(U)
  \tilde{Z}_x
  \bigr)_x
  \bigr\}_x
  \Bigr]_x,
  \nu_j(U)
  \Bigr\rangle
\nonumber
\\
  =
& -
  \left\langle
  P(U)
  \bigl\{
  \tilde{J}(U)
  P(U)
  \bigl(
  P(U)
  \tilde{Z}_x
  \bigr)_x
  \bigr\}_x,
  \frac{\partial \nu_j}{\partial U}(U)
  U_x
  \right\rangle
  \label{equation:453}
\end{align}
since 
$$
\Bigl\langle
P(U)
\bigl\{
\tilde{J}(U)
P(U)
\bigl(
P(U)
\tilde{Z}_x
\bigr)_x
\bigr\}_x,
\nu_j(U)
\Bigr\rangle
=
0.
$$
Applying 
integration by parts, 
\eqref{equation:452} 
and 
\eqref{equation:453} 
to 
\eqref{equation:410}, 
we have
\begin{align}
  \text{\eqref{equation:410}}
& =
  a
  \int_{\mathbb{R}}
  \Bigl\langle
  \Bigl[
  P(U)
  \bigl\{
  \tilde{J}(U)
  P(U)
  \bigl(
  P(U)
  \tilde{Z}_x
  \bigr)_x
  \bigr\}_x
  \Bigr]_x,
  \tilde{Z}
  \Bigr\rangle
  dx
\nonumber
\\
& -
  a
  \sum_{j=2n+1}^d
  \int_{\mathbb{R}}
  \Bigl\langle
  \Bigl[
  P(U)
  \bigl\{
  \tilde{J}(U)
  P(U)
  \bigl(
  P(U)
  \tilde{Z}_x
  \bigr)_x
  \bigr\}_x
  \Bigr]_x,
  \nu_j(U)
  \Bigr\rangle
  \bigl\langle
  \nu_j(U),\tilde{Z}
  \bigr\rangle
  dx
\nonumber
\\
& =
  -
  a
  \int_{\mathbb{R}}
  \Bigl\langle
  \bigl\{
  \tilde{J}(U)
  P(U)
  \bigl(
  P(U)
  \tilde{Z}_x
  \bigr)_x
  \bigr\}_x,
  P(U)
  \tilde{Z}_x
  \Bigr\rangle
  dx
\nonumber
\\
& +
  a
  \sum_{j=2n+1}^d
  \int_{\mathbb{R}}
  \left\langle
  P(U)
  \bigl\{
  \tilde{J}(U)
  P(U)
  \bigl(
  P(U)
  \tilde{Z}_x
  \bigr)_x
  \bigr\}_x,
  \frac{\partial \nu_j}{\partial U}(U)
  U_x
  \right\rangle
  \bigl\langle
  \nu_j(U),\tilde{Z}
  \bigr\rangle
  dx
\nonumber
\\
& =
  a
  \int_{\mathbb{R}}
  \Bigl\langle
  \tilde{J}(U)
  P(U)
  \bigl(
  P(U)
  \tilde{Z}_x
  \bigr)_x,
  \bigl(
  P(U)
  \tilde{Z}_x
  \bigr)_x
  \Bigr\rangle
  dx
\nonumber
\\
& +
  \int_{\mathbb{R}}
  \mathcal{O}
  \Bigl(
  \phi(t,x)^{1/2}
  \bigl(
  \lvert\tilde{Z}_{xxx}\rvert
  +
  \lvert\tilde{Z}_{xx}\rvert
  +
  \lvert\tilde{Z}_{x}\rvert
  \bigr)
  \Bigr)
  \mathcal{O}
  \Bigl(
  \phi(t,x)^{1/2}Z
  \Bigr)
  dx
\nonumber
\\
& =
  \int_{\mathbb{R}}
  \mathcal{O}
  \Bigl(
  \phi(t,x)
  \lvert\tilde{Z}_x\rvert^2
  +
  \lvert\tilde{Z}\rvert^2
  +
  \lvert{Z_x}\rvert^2
  +
  \lvert{Z}\rvert^2
  \Bigr)
  dx
\nonumber
\\
& \leqslant
  C_2
  \int_{\mathbb{R}}
  \phi(t,x)
  \lvert\tilde{Z}_x\rvert^2
  dx
  +
  C D(t)^2.
\label{equation:414}
\end{align}
Using integration by parts, we have 
\begin{align}
  \text{\eqref{equation:411}}
& =
  -
  \int_{\mathbb{R}}
  \bigl\langle
  \tilde{J}(U)P(U)\tilde{Z}_x,P(U)\tilde{Z}_x
  \bigr\rangle
  -
  \int_{\mathbb{R}}
  \left\langle
  \tilde{J}(U)P(U)\tilde{Z}_x,
  \frac{\partial P}{\partial U}(U)
  U_x
  \tilde{Z}
  \right\rangle
\nonumber
\\ 
& =
  \int_{\mathbb{R}}
  \mathcal{O}\Bigl(\phi(t,x)^{1/2}\lvert\tilde{Z}_x\rvert \lvert\tilde{Z}\rvert\Bigr)
  dx
\nonumber
\\
& \leqslant
  C_3
  \int_{\mathbb{R}}
  \phi(t,x)
  \lvert\tilde{Z}_x\rvert^2
  dx
  +
  C D(t)^2.
\label{equation:415}
\end{align}
Combining 
\eqref{equation:413}, 
\eqref{equation:414} 
and  
\eqref{equation:415}, 
we obtain 
\begin{equation}
\frac{d}{dt}
\frac{1}{2}
\int_{\mathbb{R}}
\lvert\tilde{Z}\rvert^2
dx
\leqslant
-
(M-C_1-C_2-C_3)
\int_{\mathbb{R}}
\phi(t,x)
\lvert\tilde{Z}_x\rvert^2
+
C D(t)^2
\leqslant
C D(t)^2
\label{equation:416}
\end{equation}
provided that $M$ is sufficiently large. 
Combining 
\eqref{equation:408}, 
\eqref{equation:409} 
and  
\eqref{equation:416}, 
we have 
$$
\frac{d}{dt}
D(t)^2
\leqslant
C D(t)^2,
\quad
D(0)^2=0,
$$
which implies that $D(t)^2\equiv0$. 
This completes the proof of the uniqueness of solutions. 
\end{proof}
Finally, we shall prove the continuity of the unique solution in time. 
\begin{proof}[{\bf Proof of Continuity in Time}]
Let $k$ be an integer not smaller than six, 
and let $u$ be a unique solution to the initial value problem satisfying 
$u \in L^\infty(0,T;H^{k+1}(\mathbb{R};TN))$.  
By using the equation \eqref{equation:pde}, 
we deduce that 
$u_t \in L^\infty(0,T;H^{k-3}(\mathbb{R};TN))$ 
and 
$u \in C([0,T];H^{k-3}(\mathbb{R};TN))$. 
Moreover, the interpolation implies that 
$u \in C([0,T];H^{k}(\mathbb{R};TN))$ and 
$u$ is $H^{k+1}$-valued continuous weakly. 
Set 
\begin{align*}
  V_k
& =
  \nabla_x^ku_x
  +
  \frac{M}{4a}
  \Phi(t,x)
  J(u)
  \nabla_x^{k-1}u_x,
\\
  \Phi(t,x)
& =
  \int_{-\infty}^x
  g\bigl(u_y(t,y),u_y(t,y)\bigr)
  dy,
\\
  \mathscr{N}(u)^2
& =
  \int_{\mathbb{R}}
  \left\{
  g(V_k,V_k)
  +
  \sum_{l=0}^{k-1}
  g\bigl(\nabla_x^lu_x,\nabla_x^lu_x\bigr)
  \right\}
  dx, 
\end{align*}
where $M$ is a positive constant determined in Section~\ref{section:energy}. 
It suffices to show that $V_k$ is $L^2(\mathbb{R};TN)$-valued continuous at $t=0$. 
It is easy to see that $V_k$ is $L^2(\mathbb{R};TN)$-valued continuous weakly and that 
$$
\int_{\mathbb{R}}
g(V_k(0),V_k(0))
dx
\leqslant
\liminf_{t\downarrow0}
\int_{\mathbb{R}}
g(V_k(t),V_k(t))
dx.
$$
By using \eqref{equation:energy3} and $u \in C([0,T];H^{k}(\mathbb{R};TN))$, we deduce  
$$
\limsup_{t\downarrow0}
\int_{\mathbb{R}}
g(V_k(t),V_k(t))
dx
\leqslant
\int_{\mathbb{R}}
g(V_k(0),V_k(0))
dx.
$$
Hence 
$$
\lim_{t\downarrow0}
\int_{\mathbb{R}}
g(V_k(t),V_k(t))
dx
=
\int_{\mathbb{R}}
g(V_k(0),V_k(0))
dx.
$$
Combining this and weak continuity, we deduce that 
$V_k$ is $L^2(\mathbb{R};TN)$-valued continuous at $t=0$. 
\end{proof}
%
%

\end{document}